\documentclass[12pt, reqno]{amsart}

\usepackage{amsmath, amssymb, amscd, amsthm, comment, amsxtra}
\usepackage{graphicx}
\usepackage{subfigure}
\usepackage{enumerate}
\usepackage[pdftex, linktocpage=true]{hyperref}
\usepackage{euscript}
\usepackage[format=plain,labelfont=bf,up,width=.99\textwidth]{caption}
\usepackage[toc,page]{appendix}
\usepackage{mathabx}
\usepackage{xcolor}
\usepackage{accents}

\theoremstyle{plain}
\newtheorem{thm}{Theorem}[section]
\newtheorem*{thma}{Theorem A}
\newtheorem*{thmb}{Theorem B}
\newtheorem{lem}[thm]{Lemma}
\newtheorem{prop}[thm]{Proposition}
\newtheorem{cor}[thm]{Corollary}

\theoremstyle{definition}

\newtheorem{rem}[thm]{Remark}

\newtheorem{notn}[thm]{Notation}

\newcommand{\diam}{\operatorname{diam}}

\newcommand{\dist}{\operatorname{dist}}

\renewcommand{\Im}{\operatorname{Im}}
\renewcommand{\mod}{\operatorname{mod}}

\newcommand{\Crit}{\operatorname{Crit}}

\newcommand{\rot}{\operatorname{rot}}
\newcommand{\crit}{\operatorname{crit}}
\newcommand{\ixp}{\operatorname{ixp}}
\newcommand{\gen}{\operatorname{gen}}

\newcommand{\In}{\operatorname{in}}
\newcommand{\ot}{\operatorname{out}}
\newcommand{\tp}{\operatorname{top}}
\newcommand{\bt}{\operatorname{bot}}

\newcommand{\bg}{\operatorname{big}}
\newcommand{\sm}{\operatorname{sm}}
\newcommand{\ext}{\operatorname{ext}}
\newcommand{\Int}{\operatorname{int}}

\newcommand{\Ang}{\operatorname{Ang}}
\newcommand{\Imp}{\operatorname{Imp}}

\numberwithin{equation}{section}
\newcommand{\thmref}[1]{Theorem~\ref{#1}}
\newcommand{\propref}[1]{Proposition~\ref{#1}}
\newcommand{\lemref}[1]{Lemma~\ref{#1}}
\newcommand{\corref}[1]{Corollary~\ref{#1}}
\newcommand{\figref}[1]{Figure~\ref{#1}}
\newcommand{\secref}[1]{Section~\ref{#1}}

\newcommand{\notref}[1]{Notation~\ref{#1}}

\newcommand{\bbN}{\mathbb N}
\newcommand{\bbZ}{\mathbb Z}
\newcommand{\bbQ}{\mathbb Q}
\newcommand{\bbR}{\mathbb R}
\newcommand{\bbC}{\mathbb C}
\newcommand{\bbD}{\mathbb D}

\newcommand{\pbbD}{{\partial \mathbb D}}

\newcommand{\hbbC}{\hat{\mathbb C}}

\newcommand{\fJ}{\mathfrak J}
\newcommand{\fU}{\mathfrak U}
\newcommand{\fr}{\mathfrak r}
\newcommand{\fd}{\mathfrak d}

\newcommand{\cE}{\mathcal E}
\newcommand{\cF}{\mathcal F}

\newcommand{\cH}{\mathcal H}
\newcommand{\cI}{\mathcal I}
\newcommand{\cJ}{\mathcal J}

\newcommand{\cL}{\mathcal L}

\newcommand{\cQ}{\mathcal Q}
\newcommand{\cR}{\mathcal R}

\newcommand{\cV}{\mathcal V}
\newcommand{\cW}{\mathcal W}

\newcommand{\cZ}{\mathcal Z}

\newcommand{\bfA}{\mathbf A}
\newcommand{\bfB}{\mathbf B}

\newcommand{\bfD}{\mathbf D}
\newcommand{\bfE}{\mathbf E}

\newcommand{\bfH}{\mathbf H}

\newcommand{\bfP}{\mathbf P}

\newcommand{\bfR}{\mathbf R}

\newcommand{\bfT}{\mathbf T}

\newcommand{\bfV}{\mathbf V}

\newcommand{\bfe}{\mathbf e}

\newcommand{\bfi}{\mathbf i}

\newcommand{\bfr}{\mathbf r}

\newcommand{\bn}{\bar n}
\newcommand{\bN}{\bar N}
\newcommand{\bm}{\bar m}

\newcommand{\bk}{\bar k}

\newcommand{\tif}{\tilde f}

\newcommand{\hA}{\hat A}

\newcommand{\hU}{\hat U}

\newcommand{\hC}{\hat C}
\newcommand{\tiF}{\tilde F}
\newcommand{\tiH}{\tilde H}
\newcommand{\tiI}{\tilde I}

\newcommand{\tiS}{\tilde S}
\newcommand{\ticE}{\tilde \cE}

\newcommand{\hg}{\hat g}
\newcommand{\chR}{\check \cR}

\newcommand{\chB}{\check B}

\newcommand{\hI}{\hat I}
\newcommand{\hn}{\hat n}

\newcommand{\hE}{\hat E}
\newcommand{\hJ}{\hat J}
\newcommand{\hV}{\hat V}

\newcommand{\hDelta}{\hat \Delta}
\newcommand{\hgamma}{\hat \gamma}
\newcommand{\tiJ}{\tilde J}
\newcommand{\tiP}{\tilde P}
\newcommand{\tiX}{\tilde X}
\newcommand{\tiE}{\tilde E}

\newcommand{\tiZ}{\tilde Z}
\newcommand{\ticZ}{\tilde \cZ}

\newcommand{\hGamma}{\hat \Gamma}
\newcommand{\chGamma}{\check \Gamma}
\newcommand{\chDelta}{\check \Delta}
\newcommand{\tiGamma}{\tilde \Gamma}

\newcommand{\hlambda}{\hat \lambda}

\newcommand{\matsp}[1]{\hspace{5mm} \text{#1} \hspace{5mm}}

\title[Local Connectivity at Polynomial Siegel Boundaries]{Local Connectivity of Polynomial Julia sets at Bounded Type Siegel Boundaries}
\author{Jonguk Yang}

\begin{document}

\begin{abstract}
Consider a polynomial $f$ of degree $d \geq 2$ that has a Siegel disk $\Delta_f$ with a rotation number of bounded type. We prove that there does not exist a hedgehog containing $\Delta_f$. Moreover, if the Julia set $J_f$ of $f$ is connected, then it is locally connected at the Siegel boundary $\partial \Delta_f$.
\end{abstract}

\maketitle


\section{Introduction}\label{sec:intro}

The object of this paper is to study polynomial dynamical systems that feature irrationally indifferent periodic points. Towards this end, let $f : \hat\bbC \to \hat\bbC$ be a polynomial of degree $d \geq 2$, and suppose that $0$ is a $p$-periodic point with an irrational rotation number $\rho \in (\bbR \setminus \bbQ)/\bbZ$. This means that the multiplier of $0$ is given by
$$
(f^p)'(0) = e^{2\pi i \rho}.
$$
By replacing $f$ by $f^p$ if necessary, we may assume without loss of generality that $0$ is a fixed point (i.e. $p=1$).

Since $f$ is a polynomial, $\infty$ is a superattracting fixed point. The {\it attracting basin of infinity} is the set of all points which converge to $\infty$ under iteration of $f$:
$$
A^\infty_f := \{z \in \hbbC \, | \, f^n(z) \to \infty \text{ as } n \to \infty\}.
$$
The {\it filled Julia set} is the set of all points whose orbits are bounded:
$$
K_f := \hbbC \setminus A^\infty_f.
$$
The {\it Julia set} is the common boundary of these two sets:
$$
J_f := \partial A^\infty_f = \partial K_f.
$$
Lastly, the {\it Fatou set} is the complement of the Julia set:
$$
F_f := \hbbC \setminus J_f = \mathring{K_f} \cup A^\infty_f.
$$
Alternately, the Fatou set $F_f$ can be defined as the domain of normality of the iterates of $f$, and the Julia set $J_f$ as the complement of $F_f$. The latter definition is more general as it also applies to the dynamics of non-polynomial rational maps.

Let $K_f^0$ be the connected component of $K_f$ containing $0$. Then we have $f(K_f^0) = K_f^0$. Moreover, it can be shown that there exists a polynomial $g$ with a connected filled Julia set $K_g$ such that $f|_{K_f^0}$ and $g|_{K_g}$ are quasiconformally conjugate (see \cite{Ki} Lemma 4.2). Hence, we may assume without loss of generality that $K_f$ is connected. This is equivalent to assuming that $\infty$ is the only critical point of $f$ contained in $A^\infty_f$.

Since $K_f$ is connected, $A^\infty_f$ is a simply connected domain. By the Riemann Mapping Theorem, there exists a unique conformal map $\phi^\infty_f : A^\infty_f \to \hbbC \setminus \overline{\bbD}$ such that $\phi^\infty_f(\infty) = \infty$ and $(\phi^\infty_f)'(\infty) =1$. Moreover, it is known (see \cite{Mi2} Theorem 9.5) that $\phi^\infty_f$ conjugates $f$ to the power map $z \mapsto z^{d}$:
$$
\phi^\infty_f \circ f \circ (\phi^\infty_f)^{-1}(z) = z^{d}
\matsp{for}
z\in \hat\bbC \setminus \overline{\bbD}.
$$
The map $\phi^\infty_f$ is called the {\it B\"ottcher uniformization} of $f$.

A Hausdorff space $X$ is {\it locally connected} at $x \in X$ if $x$ has arbitrarily small connected open neighborhoods in $X$. If this is true at every point in $X$, then $X$ is said to be {\it locally connected}. By Carath\'eodory's Theorem, the inverse B\"ottcher uniformization $(\phi^\infty_f)^{-1}$ extends to a continuous map from $\partial \bbD$ to $J_f$ if and only if $J_f$ is locally connected. In this case, the map
$$
\chi_f := (\phi^\infty_f)^{-1}|_{\partial \bbD}
$$
gives a continuous parameterization of $J_f$ by $\bbR /\bbZ \cong \partial \bbD$. Moreover, $\chi_f$ is a semi-conjugacy between $f$ and the angle $d$-tupling map $t \mapsto d t$:
$$
f \circ \chi_f(t) = \chi_f(d t)
\matsp{for}
t \in \bbR /\bbZ.
$$
Thus, $\chi_f$ can be viewed as providing a topological model of the dynamics of $f$ on $J_f$.

Typically, local connectivity of $J_f$ is proved by showing that the dynamics of $f$ is {\it combinatorially rigid}. Loosely speaking, this means that every point in $J_f$ exhibits a distinct combinatorial behavior under iteration by $f$ with respect to some suitable Markov partition of $J_f$ (called a {\it puzzle partition}).

Our discussion thus far applies generally to all polynomials. We now focus specifically on the dynamics of $f$ near the irrationally indifferent fixed point $0$. We say that $0$ is a {\it Siegel point} if $f$ is linearizable in a neighborhood of $0$. Otherwise, $0$ is called a {\it Cremer point}. It is known that the point $0$ is Siegel if and only if $0 \in F_f$. In this case, let $\Delta_f \subset F_f$ be the Fatou component containing $0$. Then the linearizing conjugacy near $0$ has a maximal extension to a conformal map $\phi^0_f : (\Delta_f, 0) \to (\bbD, 0)$ such that $|(\phi^0_f)'(0)| = 1$ and
$$
\phi^0_f \circ f \circ (\phi^0_f)^{-1}(z) = e^{2\pi i \rho} z
\matsp{for}
z\in \bbD
$$
(see \cite{Mi2} Lemma 11.1). The Fatou component $\Delta_f$ is called a {\it Siegel disk}. Clearly, the Siegel boundary $\partial \Delta_f$ is contained in $J_f$.

There exists a critical point $c \in J_f$ that either accumulates to $0 \in J_f$ if $0$ is Cremer, or to $\partial \Delta_f \subset J_f$ if $0$ is Siegel (see \cite{Mi2} Theorem 11.17). Moreover, it follows from a classical result by Ma\~n\'e \cite{Ma} that $c$ must be recurrent. In such a highly nonlinear situation, it is possible for the geometry of $J_f$ to become wildly distorted. Indeed, one can show from general principles that if $0$ is Cremer, or if $0$ is Siegel and $\partial \Delta_f$ does not contain a critical point, then $J_f$ cannot be locally connected (see \cite{Mi2} Corollary 18.6).

The dynamical behavior of $f$ near $0$ is strongly dependent on the arithmetic nature of the rotation number $\rho \in (\bbR \setminus \bbQ)/\bbZ$. Let $p_n/q_n$ be the continued fraction convergents of $\rho$ (see \secref{sec:a priori}). We say that $\rho$ is {\it Diophantine (of order $k \geq 2$)} if for some $C >0$, we have $q_{n+1} < C q_n^k$. If $k=2$, then $\rho$ is said to be {\it of bounded type}. More generally, $\rho$ is {\it Brjuno} if
$$
\sum_{n=1}^\infty \frac{\log q_{n+1}}{q_n} < \infty.
$$
Lastly, we say that $\rho$ is {\it Herman} if every analytic circle diffeomorphism with rotation number $\rho$ is analytically linearizable (an explicit arithmetical description of this property is given by Yoccoz \cite{Yo1}). Let $\bfB\bfT$, $\bfD\bfi$, $\bfH\bfe$ and $\bfB\bfr$ denote the set of bounded type, Diophantine, Herman and Brjuno numbers respectively. Then we have
$$
\bfB\bfT \subsetneq \bfD\bfi \subsetneq \bfH\bfe \subsetneq \bfB\bfr.
$$
Siegel proved that if $\rho \in \bfD\bfi$, then $0$ is a Siegel point \cite{Si}. This was generalized to $\rho \in \bfB\bfr$ by Brjuno \cite{Br}. Yoccoz showed that the converse to Brjuno's theorem is true if $f$ is a quadratic polynomial \cite{Yo2}. Finally, for $\rho \in \bfH\bfe$, Herman proved that $\partial \Delta_f$ contains a critical point if $f$ is injective on $\overline{\Delta_f}$ \cite{He1}.

Suppose $f$ has a Siegel disk $\Delta_f \ni 0$. Define a {\it Siegel continuum} $\hDelta_f$ as any compact, connected set containing $\overline{\Delta_f}$ that maps bijectively into itself by $f$. This definition can be viewed as identifying a part of the filled Julia set $K_f$ that is relevant to the local dynamics of $f$ near $\Delta_f$. Following Perez-Marco \cite{PM}, we say that $\hDelta_f$ is a {\it Siegel hedgehog} if $\hDelta_f \neq \overline{\Delta_f}$. The known examples suggest that if a Siegel hedgehog exists, then it is likely to be a highly non-locally connected set (see \cite{Chera}).

The main goal of this paper is to prove the following two results.

\begin{thma}[Local connectivity at the Siegel boundary]
Let $f : \hat\bbC \to \hat\bbC$ be a polynomial of degree $d \geq 2$ with a connected Julia set $J_f$. Suppose $f$ has a Siegel disk $\Delta_f$ whose rotation number $\rho \in (\bbR \setminus \bbQ)/\bbZ$ is of bounded type. Then $J_f$ is locally connected at every point in $\partial \Delta_f$.
\end{thma}

\begin{thmb}[No Siegel hedgehogs]
Let $f : \hat\bbC \to \hat\bbC$ be a polynomial of degree $d \geq 2$. Suppose $f$ has a Siegel disk $\Delta_f$ whose rotation number $\rho \in (\bbR \setminus \bbQ)/\bbZ$ is of bounded type. Then $\overline{\Delta_f}$ is the only Siegel continuum containing $\Delta_f$.
\end{thmb}

In general, it is not possible to extend Theorem A to the entire Julia set, as there exist polynomials with Siegel disks of bounded type whose Julia sets are not locally connected everywhere. For example, in the phase space of a cubic polynomial, a Siegel disk of bounded type can coexist with a Cremer point (see \cite{Za}).

\subsection{Background}

The importance of local connectivity in polynomial dynamics was brought to limelight by Douady-Hubbard, who initiated the modern approach to the subject in their seminal Orsay Notes \cite{DoHu}. The use of puzzles for combinatorial rigidity was pioneered by Yoccoz, who showed that the Julia set of a quadratic polynomial is locally connected if it is at most finitely renormalizable and has no indifferent periodic orbits \cite{Mi1}. In his proof, the Gr\"otzsch inequality is applied to an infinite nest of annuli to conclude that puzzle pieces shrink to singletons. Unfortunately, the basic estimate he used to obtain the required lower bound on the conformal moduli is insufficient for higher degree polynomials. In \cite{KaLy2}, Kahn-Lyubich developed a far more sophisticated version of this estimate called the Covering Lemma (see \thmref{cover lem}), which allowed them to generalize Yoccoz's result to the unicritical case of higher degree \cite{KaLy1}. Finally, Kozlovski-van Strien combined the Covering lemma with a combinatorial tool called enhanced nests to prove the following theorem.

\begin{thm}[Kozlovski-van Strien \cite{KovS}]\label{yoccoz}
Let $f : \hat\bbC \to \hat\bbC$ be a polynomial of degree $d \geq 2$ with a connected Julia set $J_f$. Suppose $f$ has no indifferent periodic orbits, and is at most finitely renormalizable (in the sense of polynomial-like mappings).  Then $J_f$ is locally connected.
\end{thm}

In \cite{RoYi}, Roesch-Yin extended the aforementioned techniques to also allow for rationally indifferent periodic orbits. With this, they were able to prove that the boundaries of bounded attracting or parabolic Fatou components of polynomials are Jordan curves.

Yoccoz's result and subsequent generalizations provide important motivation for our work, as they show that the existence of irrationally indifferent periodic orbits is one of the only two possible obstructions to local connectivity of polynomial Julia sets. In terms of techniques, the notion of puzzles and the Covering Lemma are both used in our paper in essential ways. On the other hand, the particular combinatorial aspects of the proofs (including the use of enhanced nests) are not applicable to our setting, since the underlying dynamical model is given by angle rotation rather than angle $d$-tupling.

A major early breakthrough in the study of Siegel disks was due to the combined efforts of Douady, Ghys, Herman, Shishikura and \'Swi\c atek. They discovered that a quadratic polynomial with a Siegel disk of bounded type rotation number can be modeled by a Blaschke product via quasiconformal surgery. This led to the celebrated result that the boundaries of such Siegel disks are quasi-circles containing the critical point \cite{Do}. By analyzing the surgery map defined on a certain space of some degree-5 Blaschke products, Zakeri proved that the analog of the above results is also true for cubic polynomials \cite{Za}. For polynomials of higher degree, Shishikura announced on his webpage that Siegel boundaries of bounded type rotation numbers are quasi-circles; each cycle of which contains a critical point. This was then generalized by Zhang to apply to all rational maps.

\begin{thm}[Zhang \cite{Zh}]\label{quasicircle}
Let $f : \hat \bbC \to \hat \bbC$ be a rational map of degree $d \geq 2$ with a $p$-periodic Siegel disk $\Delta_f$ of bounded type rotation number. Then $\partial \Delta_f$ is a quasi-circle whose cycle $\partial \Delta_f \cup \ldots \cup f^p(\partial \Delta_f)$ contains a critical point.
\end{thm}

\thmref{quasicircle} was proved by approximating the Siegel boundary by quasi-circles within the Siegel disk. Building on this argument, we prove that every polynomial with a Siegel disk of bounded type has a Blaschke product model (\thmref{blaschke model}).

In the phase space of a Blaschke product model, the Siegel boundary is straightened to the unit circle. This added symmetry greatly facilitates efforts to control the geometry of the system. In particular, on the unit circle itself, the geometry of the orbit structure is controlled by {\it real a priori} bounds. This was proved by Herman \cite{He2} who used estimates derived by \'Swi\c atek \cite{Sw}.

Working with the Blaschke product models of quadratic Siegel polynomials, Petersen was able to prove the following result.

\begin{thm}[Petersen \cite{Pe}]\label{quadratic}
Let $f : \hat \bbC \to \hat \bbC$ be a quadratic polynomial with a Siegel disk $\Delta_f$ of bounded type rotation number. Then its Julia set $J_f$ is locally connected.
\end{thm}

In \cite{Ya1}, Yampolsky gave an alternative proof of this result by using {\it complex a priori} bounds for unicritical circle maps. However, neither of these proofs generalize to higher degrees (not even to say, cubics), since they do not account for the existence of any critical points whose orbits are not strictly confined to the unit circle. Since then, Yampolsky, with Estevez and Smania, have generalized complex a priori bounds to multicritical circle maps with bounded type rotation numbers \cite{EsSmYa}. However, this alone does not solve the aforementioned problem, since the orbits of the relevant critical points in their system are still confined to the circle.

Theorem A of our paper is an extension of \thmref{quadratic} to higher degree polynomials (note however, that global local connectivity is no longer necessarily true when there are two or more critical points). Some aspects of the work of Petersen and Yampolsky are important in our argument as well. In particular, we use bubble rays to construct puzzle partitions similarly to Petersen. Additionally, we develop and use a softer version of complex a priori bounds (see Subsection \ref{subsec:comp ext}) to control local geometry near the unit circle similarly to Yampolsky.

Our work also connects to and builds upon the deep body of research concerning the topological structure of invariant rotation continuums centered at irrationally indifferent periodic points. What makes this topic so profoundly difficult is that the setting inherently lacks hyperbolicity---a property which has been proved so fruitful in the study of other kinds of dynamical systems. Nonetheless, there have been substantial progress due to important contributions from numerous authors (see e.g. \cite{Rog}, \cite{PM}, \cite{PeZa}, \cite{AvBuCh}, \cite{Chi}, \cite{Cheri3}, \cite{Chera}, \cite{WaYaZhZh}, and the references therein). In a remarkable recent development, Cheraghi used the results of Inou-Shishikura's near-parabolic renormalization theory \cite{InSh} to prove the following theorem.

\begin{thm}[Cheraghi \cite{Chera}]\label{cheraghi}
Let $f : \hat \bbC \to \hat \bbC$ be a quadratic polynomial with an irrationally indifferent fixed point $0$ whose rotation number $\rho \in (\bbR\setminus \bbQ)/\bbZ$ is of sufficiently high type. Denote the critical point and the Siegel disk (if it exists) of $f$ by $c$ and $\Delta_f \ni 0$ respectively. Then one of the following statements hold.
\begin{enumerate}[i)]
\item If $\rho$ is Herman, then $\omega(c) = \partial \Delta_f$ is a Jordan curve.
\item If $\rho$ is Brjuno but not Herman, then $\partial \Delta_f$ is a Jordan curve, and $\omega(c) \supsetneq \partial \Delta_f$ is a one-sided hairy Jordan curve.
\item If $\rho$ is not Brjuno, then $\omega(c) \ni 0$ is a Cantor bouquet.
\end{enumerate}
\end{thm}

As there are no known counterexamples, it is conjectured that \thmref{cheraghi} can be generalized to all rational maps $f$ and all irrational rotation numbers $\rho$. Theorem B in our paper can be viewed as verifying this conjecture in the case that $f$ is a polynomial and $\rho$ is of bounded type. In the unicritical case, if $\rho$ is Herman, it suffices to show that $\partial \Delta_f$ is a Jordan curve. Then $f|_{\Delta_f}$ extends to a homeomorphism on $\overline{\Delta_f}$, and the unique critical point must be trapped in $\partial \Delta_f$ by Herman's result \cite{He1}. However, in the multicritical case, knowing that $\partial \Delta_f$ is a Jordan curve is not enough. A priori, it is still possible to have a one-sided hairy Jordan Siegel disk, as long as one of the additional critical points is at the tip of one of the ``hairs.'' When reading our paper, it may be helpful to keep this picture in mind as the most likely counterexample that we must rule out. To the best of author's knowledge, Theorem B of our paper is the first result that establishes the topology of multicritical invariant rotation continuums within a certain rotation number class.

\subsection{Strategy of proof}

To prove the main theorems of our paper, we first model the dynamics of the polynomial $f$ by that of a Blaschke product $F$ (\secref{sec:blaschke}). In this model, the Siegel boundary $\partial \Delta_f$ is straightened to the unit circle $\pbbD$. This allows us to invoke the renormalization theory of analytic circle homeomorphisms to control the local geometry of $F$ near $\pbbD$ (\secref{sec:a priori} and \secref{sec:geometry}).

Next, we partition the phase space of $F$ into combinatorial pieces called {\it puzzles}. To do this, we use external rays inside the basin of infinity and the basin of $0$, as well as structures inside $J_F$, called {\it bubble rays}, that are constructed from preimages of $\pbbD$ (\secref{sec:puzzle} and \secref{sec:disks}). Using these puzzles, we analyze the conformal geometry of $F$ near $\pbbD$. More specifically, we form annuli using strictly nested puzzles that intersect $\pbbD$, then study how their moduli transform under the dynamics.

The key difficulty we must overcome is that puzzles intersecting $\pbbD$ break down under iteration of $F$. This is caused by the incompatibility of the combinatorics of the external rays with the combinatorics of the bubble rays. The former is governed by the angle multiplier map, while the latter is governed by the angle rotation map. As a result, the iterated images of the puzzles start to develop slits along $\pbbD$. However, using a priori bounds, we show that cutting slits into annuli that are already nearly degenerate does not significantly decrease their moduli. Supplementing this argument with the Kahn-Lyubich Covering Lemma, we are able to prove that nested puzzle annuli surrounding a point on $\pbbD$ has infinite modulus (\secref{sec:lc at crit} and \secref{sec:lc on circ}). Theorem A then follows by Gr\"otzsch inequality.

Finally, Theorem B is proved by showing that any invariant rotation continuum $\hDelta_f$ containing the Siegel disk $\Delta_f$ must be trapped inside puzzle neighborhoods intersecting $\partial \Delta_f$. Since any infinite nest of puzzle pieces in these neighborhoods must shrink to a point in $\partial \Delta_f$, it follows that $\hDelta_f \subseteq \overline{\Delta_f}$ as claimed.

\subsection*{Acknowledgement}

The author would like to thank M. Yampolsky and D. Dudko for the many helpful discussions.

\section{A Priori Bounds for Analytic Circle Maps}\label{sec:a priori}

Let $g : \partial \bbD \to \partial \bbD$ be an orientation-preserving circle homeomorphism with an irrational rotation number $\rho \in (\bbR \setminus \bbQ) /\bbZ$ (not necessarily of bounded type). Writing $\rho$ as a continued fraction, we have
\begin{equation}\label{eq:contin frac}
\rho = [a_1, a_2, \ldots] = \cfrac{1}{a_1+\cfrac{1}{a_2+ \ldots{}}}
\end{equation}
for some $a_i \in \bbN$ for $i \in \bbN$. The $a_i$'s are referred to as the {\it coefficients} of the continued fraction. Recall that $\rho$ is of bounded type if there exists a uniform bound $\tau \in \bbN$ such that $a_i \leq \tau$ for all $i \in \bbN$.

For $n \geq 2$, denote the {\it $n$th partial convergent} of $\rho$ by
$$
\frac{p_n}{q_n} := [a_1, \ldots, a_{n-1}].
$$
Letting $q_0  := 0$ and  $q_1 := 1$, it is an elementary exercise to show that the following inductive relation holds:
$$
q_n = a_{n-1}q_{n-1} + q_{n-2}.
$$

For $n \geq 1$, the number $q_n$ is referred to as the {\it $n$th closest return time}. It has the following dynamical meaning. Choose some initial point $x_0 \in \partial \bbD$, and denote $x_k := g^k(x_0)$ for $k \in \bbZ$. Define the {\it $n$th closest return arc} $I_n \subset \partial \bbD$ is the open arc with endpoints $x_0$ and $x_{q_n}$ that does not contain $x_{q_{n+1}}$. Then we have
$$
g^i(I_n) \cap (I_n \cup I_{n+1}) = \varnothing
\matsp{for}
1 \leq i < q_{n+1},
$$
and
$$
g^{q_{n+1}}(I_n) \subset I_n \cup I_{n+1}.
$$
In other words, $g^{q_{n+1}}|_{I_n}$ is the {\it first return map} of $g$ on $I_n$ to $I_n \cup I_{n+1}$.

The collection of arcs
\begin{equation}\label{eq:dyn part}
\cI_n := \{g^i(I_n) \; | \; 0\leq i < q_{n+1}\} \cup \{g^i(I_{n+1}) \; | \; 0 \leq i < q_n\}
\end{equation}
partitions $\partial \bbD$. We call $\cI_n$ the {\it $n$th dynamical partition of $\partial \bbD$}. It is easy to see that the arc $I_n$ can be partitioned into the following collection of subarcs (listed in the order they appear from $x_{q_n}$ to $x_0$):
$$
\hat \cI_{n+1} := \{g^{q_n}(I_{n+1}), g^{q_n + q_{n+1}}(I_{n+1}), \ldots, g^{q_n + (a_{n+1}-1)q_{n+1}}(I_{n+1}), I_{n+2}\}.
$$
Replacing $g^i(I_n)$ in $\cI_n$ by the images of the subarcs in $\hat \cI_{n+1}$ under $g^i$ for $i < q_{n+1}$ refines $\cI_n$ to $\cI_{n+1}$.

\subsection{Real \emph{a priori} bounds}

Henceforth, assume that the circle homeomorphism $g$ is analytic. Let $\Crit(g) \subset \pbbD$ be the finite set of critical points of $g$, and let
$$
\deg(\Crit(g)) := \{\deg(c) \; | \; c \in \Crit(g)\}.
$$

\begin{notn}
Let $I \subset \partial \bbD$ be an arc. Denote its arclength by $|I|$.
\end{notn}

In \cite{He2}, Herman proved the following geometric result about dynamic partitions of $\pbbD$ generated by analytic circle homeomorphisms (see also the translation by Ch\'eritat \cite{Cheri1}). It is based on estimates obtained by \'Swi\c atek in \cite{Sw}.

\begin{thm}[Bounded real geometry]\label{geometry bound}
Let $n \geq 0$. For each adjacent arcs $I$ and $J$ in the $n$th dynamic partition $\cI_n$, we have
$$
\frac{1}{K}|J| < |I| < K |J|,
$$
for some $K >1$ depending only on $g$. Consequently, there exist universal constants $0 < \mu_1 < \mu_2 < 1$ such that
$$
\frac{1}{K}\mu_1^n < |I_n| < K\mu_2^n.
$$
\end{thm}

\begin{cor}[Quasisymmetric conjugacy]\label{hermancor}
Suppose the rotation number $\rho \in (\bbR \setminus \bbQ)/\bbZ$ is of bounded type. Then there exists $K>1$, and a $K$-quasisymmetric homeomorphism $h : \partial \bbD \to \partial \bbD$ such that
$$
h \circ g \circ h^{-1}(z) = e^{2\pi i \rho} z
\matsp{for}
z \in \partial \bbD.
$$
\end{cor}

The proof of \thmref{geometry bound} involves controlling the distortions of $g$ along the orbits of the closest return arcs. To state this result, it is convenient to lift the action of $g$ on $\pbbD$ to the real line $\bbR$.

Define $\ixp(z) := e^{2\pi iz}$. Then $\ixp$ is a covering map from $(\bbC, 0)$ to $(\bbC^*, 1)$. In particular, we have $\ixp(\bbR) = \pbbD$. Let $\hg :\bbR \to \bbR$ be the lift of $g : \pbbD \to \pbbD$ via $\ixp$ such that
$$
g \circ \ixp (x) = \ixp \circ \hg(x)
\matsp{for}
x\in \bbR,
$$
and $\hg(0) \in (0, 1)$. Then $\hg(x+m) = \hg(x)+m$ for $x \in \bbR$ and $m \in \bbZ$, and
$$
\rho = \lim_{n \to \infty}\frac{\hg^n(x) - x}{n}.
$$
For $n \geq 1$, let $\hI_n$ be the open interval in $\bbR$ with one endpoint at $0$ such that $\ixp$ maps $\hI_n$ to the $n$th closest return arc $I_n \subset \pbbD$. We refer to $\hI_n$ as the {\it $n$th closest return interval}.

A {\it power map} $P : \bbC \to \bbC$ of {\it degree $d\in\bbN$} is given by
$$
P(z) := (z-a)^d+b
\matsp{for}
z \in \bbC,
$$
where $a, b \in \bbC$.  We say that $P$ is {\it real} if $a, b \in \bbR$, and {\it odd} if $d$ is odd. Real odd power maps restrict to homeomorphisms of $\bbR$.

Let $I \subset \bbR$ be an interval, and let $\phi : I \to \phi(I) \subset \bbR$ be an orientation-preserving diffeomorphism. We say that $\phi$ has {\it $K$-bounded distortion} for some $K > 0$ if
$$
\frac{1}{K} \leq \frac{\phi'(x)}{\phi'(y)} \leq K
\matsp{for all}
x,y \in I.
$$

Recall that the first return map of $g$ on the $n$th closest return arc $I_n \in \cI_n$ is given by $g^{q_{n+1}}|_{I_n}$. \thmref{geometry bound} is a consequence of the following result proved in \cite{He2}.

\begin{thm}[Bounded real distortion]\label{distortion bound}
For $n \geq 1$ and $0 \leq i \leq q_{n+1}$, the iterate $\hg^i$ restricted to the $n$th closest return interval $\hI_n \subset \bbR$ factors into a composition of the form:
\begin{equation}\label{eq:real factor}
\hg^i|_{\hI_n} =\phi_0 \circ P_1 \circ  \phi_1 \circ \ldots \circ P_l \circ \phi_l,
\end{equation}
where $P_k$ is a real odd power map of degree $d_k \in \deg(\Crit(g))$, and $\phi_k$ is a real analytic diffeomorphism. Moreover, $l \leq 2\#\Crit(g)$, and the distortion of $\phi_k$ is uniformly bounded independently of $n$ and $i$.
\end{thm}

\thmref{geometry bound} and \thmref{distortion bound} are collectively referred to as real {\it a priori} bounds.

\subsection{Complex extensions}\label{subsec:comp ext}

The first return map $g^{q_{n+1}}|_{I_n}$ of $g$ on the $n$th closest return arc $I_n \subset \pbbD$ extends analytically to a neighborhood of $I_n$ in $\bbC$. Loosely speaking, we say that $g$ has complex {\it a priori} bounds if the modulus of the fundamental annulus of this extension has a uniform lower bound independent of $n$. This estimate was obtained for uni-critical circle maps by Yampolsky in \cite{Ya1}, and for multi-critical circle maps with bounded type rotation numbers by Estevez, Smania and Yampolsky in \cite{EsSmYa}. 

Complex a priori bounds provides strong control over the small-scale geometry of complex extensions of analytic circle maps. However, for our application, we only need a softer version of this result, which we formulate and prove below for all irrational rotation numbers.

\begin{notn}\label{neighborhood}
Let $S \subset \bbC$. For $r >0$, denote the $r$-neighborhood of $S$ in $\bbC$ by
$$
N_r(S) :=  \{z \in \bbC \; | \; \dist(z, S) < r\}.
$$
\end{notn}

Let $I \subset \bbR$ be a compact interval, and let $\phi : I \to \phi(I) \subset \bbR$ be a real analytic diffeomorphism. We say that $\phi$ has an {\it $\eta$-complex extension} for some $\eta > 0$ if $\phi$ extends to a conformal map on $N_{\eta|I|}(I)$.

\begin{thm}[Uniform complex extension]\label{complex extension}
There exists a uniform constant $\eta >0$ independent of $n$ and $i$ such that the analytic diffeomorphisms $\phi_k$'s in \eqref{eq:real factor} have $\eta$-complex extensions.\footnote{That \thmref{complex extension} does not follow immediately from real {\it a priori} bounds was pointed out to me by D. Dudko and M. Lyubich.}
\end{thm}

To prove \thmref{complex extension}, first observe that the lifted map $\hg : \bbR \to \bbR$ extends analytically to a neighborhood $U$ of $\bbR$ in $\bbC$ such that $\Crit(\hg|_U) = \Crit(\hg|_\bbR)$, and $V := \hg(U) \supset \bbR$ contains a horizontal strip $N_R(\bbR) = \{|\Im(z)| < R\}$ for some $R >0$. Moreover, if $r >0$ is sufficiently small, then for any $c \in \Crit(\hg)$, the restriction of $\hg$ to the $r$-neighborhood $N_r(c) = \{|z-c|<r\} \subset \bbC$ of $c$ factors into the composition
\begin{equation}\label{eq:power decomp}
\hg|_{N_r(c)} = P_c \circ \psi_c,
\end{equation}
where $P_c$ is a real odd power map of degree $\deg(c)$, and $\psi_c$ is a conformal map on $N_r(c)$.

Recall that the endpoints of the arc $I_n \subset \pbbD$ are $x_0$ and $x_{q_n} := g^{q_n}(x_0)$, where $x_0 \in \pbbD$ is some given point. Let $I'_n \Supset I_n$ be the arc with endpoints $x_{q_{n+1}}$ and $x_{q_n+q_{n+2}}$ that does not contain $x_{q_{n-1}}$. Denote by $\hI'_n \Supset \hI_n$ the lift of $I'_n$ such that $\ixp(\hI'_n) = I'_n$.

Let $\cI$ be a collection of arcs in $\partial\bbD$. The {\it intersection multiplicity} of $\cI$ is the maximum number of arcs in $\cI$ whose interiors have a nonempty intersection. The following result is elementary.

\begin{lem}\label{inverse orbit intersect}
The intersection multiplicity of $\{g^j(I'_n)\}_{j=0}^{q_{n+1}-1}$ is $2$. Consequently, for $n$ sufficiently large, every critical point $c \in \Crit(g)$ is contained in at most two elements in $\{g^j(I'_n)\}_{j=0}^{q_{n+1}-1}$.
\end{lem}

Following \cite{EsSmYa}, consider the inverse orbit of $J_0 := \hg^{q_{n+1}}(\hI_n)$:
$$
\cJ_n:= \{J_{-j} := \hg^{q_{n+1}-j}(\hI_n) \; | \; 0 < j \leq q_{n+1}\}.
$$
This inverse orbit is compactly contained in the inverse orbit of $J'_0 := \hg^{q_{n+1}}(\hI'_n)$:
$$
\cJ'_n:= \{J'_{-j} := \hg^{q_{n+1}-j}(\hI'_n) \; | \; 0 < j \leq q_{n+1}\}.
$$
By \thmref{geometry bound}, there exists a uniform constant $K >1$ such that the two components of $J'_{-j} \setminus J_{-j}$ are $K$-commensurate in length to $J_{-j}$ for $0 \leq j \leq q_{n+1}$. Moreover, by \lemref{inverse orbit intersect}, there exists a sequence $0 \leq j_1 < \ldots < j_l < q_{n+1}$ with $l \leq 2 \#\Crit(g)$ such that $J'_{-j-1} \in \cJ'_n$ contains a critical point $c_k \in \Crit(\hg)$ if and only if $j = j_k$ for some $1 \leq k\leq l$.

Let $j_{l+1} := q_{n+1}$. For $1 \leq k \leq l$, we have
$$
\hg^{j_{k+1} - j_k}(J'_{-j_{k+1}}) = J'_{-j_k}.
$$
We may assume the map $\phi_k$ in \thmref{distortion bound} extends to a real analytic diffeomorphism on $J'_{-j_{k+1}}$ such that
$$
\phi_k = \psi_k \circ \hg^{j_{k+1} - j_k-1}|_{J'_{-j_{k+1}}},
$$
where $\psi_k := \psi_{c_k}$ is given in \eqref{eq:power decomp}. Denote $P_k := P_{c_k}$. Then
$$
\hg^{j_{k+1} - j_k}|_{J'_{-j_{k+1}}} = P_k \circ \phi_k.
$$
Lastly, we can assume that $\phi_0$ extends to a real analytic diffeomorphism on $J'_{-j_1}$ such that $\phi_0 = \hg^{j_1}|_{J'_{-j_1}}$.

The {\it Poincar\'e neighborhood} of an interval $I \Subset \bbR$ of hyperbolic radius $r >0$ is defined as the set of points in $\bbC|_I := (\bbC \setminus \bbR)\cup I$ whose hyperbolic distance in $\bbC|_I$ to $I$ is less than $r$. It turns out that this set is given by the $\bbR$-symmetric union of two Euclidean disks whose intersection with $\bbR$ is equal to $I$. It is clear that these disks are determined uniquely by the external angle $\theta \in (0, \pi)$ between their boundaries and $\bbR$. Henceforth, we denote the Poincare neighborhood of $I$ with external angle $\theta$ by $D_\theta(I)$. It is easy to see that
\begin{enumerate}[i)]
\item $D_{\theta_1}(I) \supset D_{\theta_2}(I)$ if $\theta_1 < \theta_2$;
\item $D_\theta(I)$ converges to $\bbC|_I$ and $I$ as $\theta$ goes to $0$ and $\pi$ respectively; and
\item $D_{\pi/2}(I)$ is a single Euclidean disk of diameter $|I|$.
\end{enumerate}

Let $I \Subset \bbR$ be an interval, and let $\phi : \bbC|_I \to \bbC|_{\phi(I)}$ be a real analytic map. Then by Schwarz lemma, $\phi(D_\theta(I)) \subset D_\theta(\phi(I))$ for any $\theta \in (0, \pi)$. This statement does not directly apply to inverse branches of $\hg$, since they do not extend globally to the entire double-slit plane. However, if the base intervals are sufficiently small, then we still have the following quasi-invariance of Poincar\'e neighborhoods (see Lemma 4.4 in \cite{Ya2}).

\begin{lem}\label{quasi inv}
Let $I \Subset \bbR$ be an interval such that $\hg^{-1}$ maps $I$ diffeomorphically onto $\hg^{-1}(I)$. Then there exist $\delta = \delta(|I|) \in (0, \pi)$ with $|I|/\delta(|I|) \to 0$ as $|I| \to 0$, and $\kappa \in (1, 2)$ such that for $\theta \in (\delta, \pi)$ and $0<\tilde \theta \leq \theta(1-|I|^\kappa)$, the inverse map $\hg^{-1}|_I$ extends analytically to a conformal map on $D_{\theta}(I)$, and $\hg^{-1}(D_\theta(I)) \subset D_{\tilde \theta}(\hg^{-1}(I))$.
\end{lem}

By \thmref{geometry bound}, the maximum length of an interval in the inverse orbit $\cJ'_n$ goes to $0$ as $n$ goes to $\infty$. Combining this fact with \lemref{quasi inv}, and then using induction, we obtain the following result (see Lemma 3.4 in \cite{dFdM}).

\begin{lem}\label{quasi inv iterate}
There exist $K_n > 1$ and $\delta_n \in (0, \pi)$ with $K_n \to 1$ and $\delta_n \to 0$ as $n \to \infty$ such that the following holds. Let $0 \leq j < i \leq q_{n+1}$ be such that $\hg^{-(i-j)}$ maps $J'_{-j}$ diffeomorphically to $J'_{-i}$. Then for $\theta \in (\delta_n, \pi)$ and $0<\tilde \theta \leq \theta/K_n$, the inverse iterate $\hg^{-(i-j)}|_{J'_{-j}}$ extends analytically to a conformal map on $D_\theta(J'_{-j})$, and $\hg^{-(i-j)}(D_\theta(J'_{-j})) \subset D_{\tilde \theta}(J'_{-i})$.
\end{lem}

The last result we need for the proof of \thmref{complex extension} is the following observation (which follows immediately from the quasisymmetry of the power map, and the quasi-invariance of sufficiently small Poincare neighborhoods under a conformal map).

\begin{lem}\label{quasi inv power}
There exist $\tilde\delta_n \in (0, \pi)$ with $\tilde \delta_n \to 0$ as $n \to \infty$ such that the following holds. For $1 \leq k \leq l$ and $\tilde \theta \in (\tilde \delta_n, \pi)$, let $W := P_k^{-1}(D_{\tilde\theta}(J'_{-j_k}))$.
Then $\psi_k^{-1}$ is defined and conformal on $W$, and there exists a constant $C = C(\tilde\theta) >1$ such that if $0 < \theta \leq \tilde\theta/C$, then $\psi_k^{-1}(W) \subset D_\theta(J'_{-j_k-1})$.
\end{lem}

\begin{proof}[Proof of \thmref{complex extension}]
Consider the constants $\delta_n, \tilde \delta_n \in (0, \pi)$, $K_n>1$, and $C(\tilde \theta)>1$ for $\tilde \theta \in (\tilde \delta_n, \pi)$ given in \lemref{quasi inv iterate} and \lemref{quasi inv power}. Let $\theta_0 := \pi/2$, and for $0\leq k \leq l$, let
$$
\tilde \theta_{k+1} := \theta_k/K_n
\matsp{and}
\theta_{k+1} := \frac{\tilde \theta_{k+1}}{C(\tilde \theta_{k+1})}.
$$
Then for $n$ sufficiently large, we have $\theta_k \in (\delta_n, \pi)$ and $\tilde \theta_k \in (\tilde \delta_n, \pi)$.

To prove the result, it suffices to show that there exist simply-connected neighborhoods $U_{k+1} \supset J'_{-j_{k+1}} \Supset J_{-j_{k+1}}$ for $0 \leq k \leq l$ such that $\phi_k$ extends to a conformal map on $U_{k+1}$, and the modulus of the annulus $A_{k+1} = U_{k+1} \setminus J_{-j_{k+1}}$ is uniformly bounded below. Choose
$U_0 := D_{\theta_0}(J'_0)$. The inverse $\phi_0^{-1}$ extends to a conformal map on $U_0$. Let
$$
U_1 := \phi_0^{-1}(U_0) \subset D_{\tilde \theta_1}(J'_{-j_1}).
$$
Proceeding inductively, assume that $U_k \subset D_{\tilde \theta_k}(J'_{-j_k})$ is defined for $1 \leq k \leq l$. Let $W_k := P_k^{-1}(U_k)$. The inverse $\psi_k^{-1}$ is defined and conformal on $W_k$, and $\psi_k^{-1}(W_k) \subset D_{\theta_{k+1}}(J_{-j_k-1})$. It follows that $\hg|_{J'_{-j_k-1}}^{-(j_{k+1} - j_k-1)}$ extends to a conformal map on $\psi_k^{-1}(W_k)$, and
$$
U_{k+1} := \phi_k^{-1}(W_k) = \hg^{-(j_{k+1} - j_k-1)} \circ \psi_k^{-1}(W_k) \subset D_{\tilde\theta_{k+1}}(J'_{-j_{k+1}}).
$$

It remains to check that the modulus of $U_k \setminus J_{-j_k}$ for $0 \leq k \leq l+1$ is uniformly bounded below. By \thmref{geometry bound}, this is true for $k = 0$. The case $k >0$ follows immediately from the fact that conformal modulus is quasi-invariant under an analytic map with uniformly bounded degree (see \lemref{pullback mod}).
\end{proof}

\subsection{Geometry near the real line}

Let $\gamma \subset \bbC$ be a simple smooth curve. We say that its slope is bounded absolutely from below by $\mu >0$ if $\gamma$ can be parameterized as $\gamma(x) = x + iy(x)$ for $x \in (a,b) \subset \bbR$ such that $\mu < |y'(x)| \leq +\infty$.

Let $I \Subset \bbR$ be an interval, and let $h : I \to h(I)$ be a real analytic map. Suppose that $h$ factors into
$$
h = \phi_0 \circ P_1 \circ \phi_1 \circ \ldots \circ P_l \circ \phi_l,
$$
where $P_k$ is a real odd power map of degree $d_k \leq D \in 2\bbN+1$, and $\phi_k$ is a real analytic diffeomorphism that has $\eta'$-complex extension for some $\eta' >0$. Denote $J_0 := h(I)$, $J_{-1} := \phi_0^{-1}(J_0)$, and
$$
\tiJ_{-k-1} := (P_k|_\bbR)^{-1}(J_{-k})
\matsp{and}
J_{-k-1} := \phi_k^{-1} (\tiJ_{-k-1})
\matsp{for}
1\leq k \leq l.
$$

\begin{prop}[Bounded geometry near the real line]\label{ext a priori}
There exist $\eta, \mu >0$ depending only on $l$, $D$ and $\eta'$ such that $h$ extends to an analytic map on $U := N_{\eta|I|}(I)$, we have $\Crit(h|_U) = \Crit(h|_I)$, and for each connected component $\gamma$ of $h^{-1}(J_0) \setminus I$, its slope is bounded absolutely from below by $\mu$.
\end{prop}

\begin{proof}
For $1 \leq k < l$, let $h_k : J_{-k} \to J_0$ be the partial composition
$$
h_k = \phi_0 \circ P_1 \circ \phi_1 \circ \ldots \circ P_{k-1} \circ \phi_{k-1}.
$$
Clearly, $h_k$ extends to an analytic map on $U_k := N_{\eta|J_{-k}|}(J_{-k})$ for some $\eta>0$ such that $\Crit(h_k|_{U_k}) = \Crit(h_k|_{J_{-k}})$. Proceeding by induction, assume that the second assertion of the lemma is true for $h_k$. Denote $X_k := h_k^{-1}(J_0)$ and $W_{k+1} := \phi_k(U_{k+1})$. By quasisymmetry of the power map, we see that a connected component $\gamma$ of $(P_k^{-1}(X_k) \cap W_{k+1}) \setminus \tiJ_{-k-1}$ have slope bounded absolutely from below by some uniform constant. By decreasing $\eta$ if necessary, we can assume that there exists a uniform constant $\epsilon >0$ such that $\gamma$ is contained in $N := N_{\epsilon|\tiJ_{-k-1}|}(x)$ for some $x \in \tiJ_{-k-1}$. By Koebe distortion theorem, $\phi_k^{-1}|_N$ approaches a linear map with scaling factor $(\phi_k'(x))^{-1} \in \bbR$ as $\epsilon \to 0$. It follows that the slope of $\phi_k^{-1}(\gamma)$ is likewise bounded absolutely from below by some uniform constant.
\end{proof}

\section{Blaschke Product Model}\label{sec:blaschke}

A rational map $F : \hbbC \to \hbbC$ which maps the circle $\pbbD$ to itself is called a {\it Blaschke product}. Let $d \geq 2$, and let $\rho \in (\bbR \setminus \bbQ) /\bbZ$ be of bounded type. Define the {\it Herman Blaschke family} $\cH^d_\rho$ of degree $2d-1$ and rotation number $\rho$ as the class of all Blaschke products of the form
\begin{equation}\label{eq:herman}
F(z) = \lambda z^d \prod_{i=1}^{d-1}\frac{1-\overline{a_i}z}{z-a} 
\end{equation}
such that
\begin{enumerate}[i)]
\item $|a_i| < 1$ for all $1 \leq i \leq d-1$,
\item $|\lambda|=1$, and
\item $g := F|_{\partial \bbD} : \partial \bbD \to \partial \bbD$ is a circle homeomorphism with rotation number $\rho$.
\end{enumerate}
In \cite{He3}, Herman proved the following result about this family (see also the translation by Ch\'eritat \cite{Cheri2}).

\begin{thm}[Uniform quasisymmetry constant]\label{hermanthm}
There exists a uniform constant $K >1$ depending only on $d$ and $\rho$ such that for every $F \in \cH^d_\rho$, there exists a $K$-quasisymmetric homeomorphism $h : \partial \bbD \to \partial \bbD$ such that
$$
h \circ F \circ h^{-1}(z) = e^{2\pi i \rho} z
\matsp{for}
z \in \partial \bbD.
$$
\end{thm}

\thmref{hermanthm} based on the following compactness result (see \cite{Zh} for the proof).

\begin{prop}\label{herman compact}
There exists a uniform constant $0 < r < 1$ depending only on $d$ and $\rho$ such that the following statements hold.
\begin{enumerate}[i)]
\item If $F \in \cH^d_\rho$, then $F$ is holomorphic on $U_r := \{|z| > r\}$.
\item For every sequence $\{F_n\}_{n=1}^\infty \subset \cH^d_\rho$, there exists a subsequence $\{F_{n_k}\}_{k=1}^\infty$ that converges compact uniformly on $U_r$ to some $F_\infty \in \cH^{d'}_\rho$ with $d' \leq d$.
\end{enumerate}
\end{prop}

Let $F \in \cH^d_\rho$. Since $F(\infty) = \infty$, and $F'(\infty) = 0$, the point $\infty$ is a superattracting fixed point of $F$. Let $A^\infty_F$ be the attracting basin of infinity. The immediate basin of infinity $\hA^\infty_F \subset \hbbC \setminus \overline{\bbD}$ is the connected component of $A^\infty_F$ containing $\infty$. The Julia set of $F$ is $J_F := \partial A^\infty_F$. Define the {\it modified Julia set} of $F$ as $\hJ_F := \partial \hA^\infty_F \subset J_F$.

The motivation for introducing the Herman Blaschke family is that it contains models of Siegel polynomials for which the dynamics on the Siegel boundaries are replaced by the dynamics of analytic circle homeomorphisms. The correspondence between Blaschke product models and Siegel polynomials are given by the quasiconformal surgery, known as the {\it Douady-Ghys surgery}, described below.

Let $h : \partial \bbD \to \partial \bbD$ be the homeomorphism given in \thmref{hermanthm}. Since $h$ is $K$-quasisymmetric, it can be extended to a $K$-quasiconformal homeomorphism on $\bbD$ such that
$$
h \circ F \circ h^{-1}(z) = e^{2\pi i \rho}z
\matsp{for}
z \in \bbD.
$$
Denote $\rot_\rho(z) := e^{2\pi i \rho}z$. Define a {\it modified Blaschke product} $\tiF : \hat\bbC \to \hat\bbC$ by
$$
\tiF(z) := \left\{
\begin{array}{cl}
h^{-1} \circ \rot_\rho \circ h(z) & : z \in \bbD \\
F(z) & : z \in \hat \bbC \setminus \bbD.
\end{array}
\right.
$$
Since $F^{-1}(\infty)\cap (\hbbC \setminus \bbD) = \{\infty\}$, we have $\tiF^{-1}(\infty) = \{\infty\}$. Moreover, the attracting basin  of $\infty$ for $\tiF$ is equal to the immediate basin $\hA^\infty_F$ of $\infty$ for $F$, and $\tiF|_{\overline{\hA^\infty_F}} \equiv F|_{\overline{\hA^\infty_F}}$. This implies that $\tiF$ is a topological polynomial of degree $d$. Define its Julia set as $J_{\tiF} := \hJ_F = \partial \hA^\infty_F \subset J_F$.

To turn $\tiF$ into an analytic polynomial, we need to find a complex structure $\sigma$ on $\bbC$ which is preserved by $\tiF$. On $\bbD$, let $\sigma$ be the pull back of the standard structure $\sigma_0$ by $h$. Since $\rot_\rho$ preserves $\sigma_0$, we see that $\tiF$ preserves $\sigma$ on $\bbD$. For $n \geq 1$, extend $\sigma$ to $\tiF^{-n}(\bbD)$ as the pull back of $\sigma|_\bbD$ by $\tiF^n$. Since $\tiF$ is holomorphic outside of $\bbD$, this does not increase the dilatation of $\sigma$. Finally, define $\sigma$ on the rest of $\bbC$ (which includes $\hA^\infty_F$) as the standard structure $\sigma_0$. It is clear from the construction that the dilatation of $\sigma$ is bounded by $K$, and that $\sigma$ is preserved under $\tiF$. By the Measurable Riemann Mapping Theorem, there exists a $K$-quasiconformal map $\eta : \bbC \to \bbC$ fixing $0$ such that $\eta^*(\sigma_0) = \sigma$.

Let
$$
f = \eta \circ \tiF \circ \eta^{-1}.
$$
Then $f$ preserves the standard complex structure $\sigma_0$. Hence it is an analytic polynomial of degree $d$. Moreover, $f$ has a Siegel disk $\Delta_f = \eta(\bbD)$ containing a Siegel fixed point $0$ of rotation number $\rho$. Observe that $\eta$ maps $(\hA^\infty_F, \infty)$ conformally onto $(A^\infty_f, \infty)$. Hence, the Julia set $J_f$ of $f$ is equal to $\eta(\hJ_F)$.

From the above discussion, we conclude that every Herman Blaschke product models a Siegel polynomial. The converse is given by the following theorem.

\begin{thm}[Existence of Blaschke product model]\label{blaschke model}
Let $f$ be a polynomial of degree $d$ that has a Siegel disk $\Delta_f$ containing a Siegel fixed point $0$ of bounded type rotation number $\rho$. Then there exist a Blaschke product model $F \in \cH^d_\rho$, the modified Blaschke product $\tiF$ obtained from $F$, and a $K$-quasiconformal map $\eta$ obtained via the Douady-Ghys surgery with $K$ given in \thmref{hermanthm} such that
$$
f = \eta \circ \tiF \circ \eta^{-1},
$$
and $\eta$ maps $(\bbD, 0)$ to $(\Delta_f, 0)$ and $(\hA^\infty_F, \infty)$ to $(A^\infty_f, \infty)$ (the latter conformally).
\end{thm}

\begin{proof}
Let $\phi : \Delta_f \to \bbD$ be a conformal map conjugating $f|_{\Delta_f}$ to the rigid rotation $\rot_\rho$ by angle $\rho$. Denote
$$
\Delta_f^t := \phi^{-1}(\{|z| < t\}).
$$
Given $0 < r <1$, choose $0 < a < r < b < 1$ so that $0 \in \Delta_f^a  \Subset \Delta_f^r \Subset \Delta_f^b \Subset \Delta_f$.

Define $X_r : \bbC \to \bbC$ as follows. Let $X_r|_{\bbC \setminus \overline{\bbD}}$ be the Riemann map onto $\bbC \setminus \overline{\Delta_f^r}$. Denote
$$
\Gamma_b := X_r^{-1}(\partial \Delta_f^b).
$$
Let $\Gamma_b^*$ be the reflection of $\Gamma_b$ about $\partial \bbD$, and let $D_b^* \Subset D_b$ be the topological disks containing $0$ bounded by $\Gamma_b^*$ and $\Gamma_b$ respectively. Define $X_r|_{D_b^*}$ as the Riemann map onto $\Delta_f^a$. Lastly, extend $X_r$ to the annulus
$$
A_b := \overline{D_b} \setminus D_b^*
$$
as a smooth map. Then $X_r$ is $K_r$-quasiconformal for some $1< K_r < \infty$ (although we may have $K_r \to \infty$ as $r \to 1$).

Define
$$
\tif_r(z):= \left\{
\begin{array}{cl}
X_r^{-1} \circ f \circ X_r(z) & : z \in \hat \bbC \setminus \bbD \\
(X_r^{-1} \circ f \circ X_r(z^*))^* & : z \in \bbD.
\end{array}
\right.
$$
where $z^*$ denotes the reflection of $z$ about $\partial \bbD$. Observe that $\tif_r : \hbbC \to \hbbC$ is a degree $2d-1$ branched covering map which is symmetric about $\partial \bbD$. Moreover, restricted to the set
$$
\tiH_r := X_r^{-1}(\Delta_f \setminus \Delta_f^r) \cup (X_r^{-1}(\Delta_f \setminus \Delta_f^r))^* \supset A_b,
$$
the map $\tif_r$ is conformally conjugate to the rigid rotation $\rot_\rho$. 

In \cite{Zh}, it is shown that there exists a $K_r$-quasiconformal map $\xi_r : \bbC \to \bbC$ mapping $(\bbD, 0)$ to $(\bbD, 0)$ such that the map
$$
F_r := \xi_r \circ \tif_r \circ \xi_r^{-1}
$$
is a Blaschke product in $\cH^d_\rho$, and $H_r := \xi_r(\tiH_r)$ is a Herman ring for $F_r$. Furthermore, it is clear by construction that $\xi_r \circ X_r^{-1}$ maps $(A^\infty_f, \infty)$ conformally onto $(\hA^\infty_{F_r}, \infty)$.

By \thmref{hermanthm}, there exists a $K$-quasiconformal map $h_r : \bbD \to \bbD$ such that 
$$
h_r \circ F_r \circ h_r^{-1}(z) = \rot_\rho(z)
\matsp{for}
z \in \bbD.
$$
The modified Blaschke product $\tiF_r$ is given by
$$
\tiF_r(z) := \left\{
\begin{array}{cl}
h_r^{-1} \circ \rot_\rho \circ h_r(z) & : z \in \bbD \\
F_r(z) & : z \in \hat \bbC \setminus \bbD.
\end{array}
\right.
$$
In \cite{Zh}, it is shown that there exists a $K$-quasiconformal conjugacy $\eta_r : \bbC \to \bbC$ mapping $(\bbD, 0)$ to $(\Delta_f^r, 0)$ and $(\hA^\infty_{F_r}, \infty)$ to $(A^\infty_f, \infty)$ (the latter conformally) such that
$$
f = \eta_r \circ \tiF_r \circ \eta_r^{-1}.
$$

By compactness of $K$-quasiconformal maps and the space $\cH^d_\rho$ (see \propref{herman compact}), we can choose $r_n \to 1$ as $n \to \infty$ such that following holds:
\begin{itemize}
\item $F_{r_n}$ converges to a Blaschke product $F \in \cH^{d'}_\rho$ for some $d' \leq d$;
\item $h_{r_n}$ converges to a $K$-quasiconformal map $h : \bbD \to \bbD$ fixing $0$;
\item $\tiF_{r_n}$ converges to the modified Blaschke product
$$
\tiF(z) := \left\{
\begin{array}{cl}
h^{-1} \circ \rot_\rho \circ h(z) & : z \in \bbD \\
F(z) & : z \in \hat \bbC \setminus \bbD
\end{array}
\right.
\matsp{;}
\text{and}
$$
\item $\eta_r$ converges to a $K$-quasiconformal map $\eta : \bbC \to \bbC$ that maps $(\bbD, 0)$ to $(\Delta_f, 0)$ and $(\hA^\infty_F, \infty)$ to $(A^\infty_f, \infty)$ (the latter conformally).
\end{itemize}
Finally, since
$$
f = \eta \circ \tiF \circ \eta^{-1},
$$
we have $d' = d$.
\end{proof}

Since $\eta$ in \thmref{blaschke model} gives a homeomorphism between $\hJ_F = \partial \hA^\infty_F$ and $J_f = \partial A^\infty_f$, we have the following result.

\begin{cor}\label{blaschke instead}
Let $f$ be a Siegel polynomial that has a Blaschke product model $F \in \cH^d_\rho$. Then the Julia set $J_f$ of $f$ is locally connected at every point in the Siegel boundary $\partial \Delta_f$ if and only if the modified Julia set $\hJ_F = \hA^\infty_F$ of $F$ is locally connected at every point in $\pbbD$.
\end{cor}

By \corref{blaschke instead}, it suffices to prove Theorem A stated in \secref{sec:intro} for the modified Julia set $\hJ_F$ of the Blaschke product $F \in \cH^d_\rho$ rather than for the Julia set $J_f$ of the Siegel polynomial $f$.

\section{Puzzle Partition}\label{sec:puzzle}

Let $\rho \in (\bbR\setminus \bbQ)/\bbZ$ be of bounded type, and let $F \in \cH^d_\rho$ be a Herman Blaschke product of the form \eqref{eq:herman} that has a critical point at $1$. Recall that the Julia set $J_F$ and the modified Julia set $\hJ_F$ of $F$ are equal to the boundary of the attracting basin of infinity $A^\infty_F$ and the immediate basin of infinity $\hA^\infty_F \subset \hbbC \setminus \overline{\bbD}$ respectively. Note that $1 \in \partial \bbD \subset \hJ_F \subset J_F$.

The restriction $g := F|_{\partial \bbD}$ is an analytic circle homeomorphism with rotation number $\rho$. Let $h : (\partial \bbD, 1) \to (\partial \bbD, 1)$ be the quasisymmetric homeomorphism given in \thmref{hermanthm} such that
$$
h \circ g \circ h^{-1}(z) = e^{2\pi \rho i}z
\matsp{for}
z \in \partial \bbD.
$$
For $s \in \bbR/\bbZ$, let
$$
\xi_s := h^{-1}(e^{2\pi s i}).
$$
For $k \in \bbZ$, denote
$$
c_k := g^k(1) = \xi_{k\rho}.
$$
Without loss of generality, we may assume that $c_k$ is not a critical point for $k \geq 1$.

Assume that $\infty$ is the only critical point in $\hA^\infty_F$, so that $J_F$ and $\hJ_F$ are connected. Then the B\"ottcher uniformization $\phi^\infty_F : \hA^\infty_F \to \hat \bbC \setminus \overline{\bbD}$ of $F$ is conformal.

The {\it external ray} of $F$ with {\it external angle} $t \in \bbR /\bbZ$ is defined as
$$
\cR^\infty_t := \{(\phi^\infty_F)^{-1}(re^{2\pi t i}) \; | \; 1 < r < \infty\}.
$$
An {\it equipotential curve} at {\it level} $l \in (1, \infty)$ of $F$ is defined as
$$
\cQ_l := \{(\phi^\infty_F)^{-1}(le^{2\pi t i}) \; | \; t \in \bbR/\bbZ\}.
$$
We have
$$
F(\cR^\infty_t) = \cR^\infty_{dt}
\matsp{and}
F(\cQ_l) = \cQ_{l^d}.
$$
We say that $\cR^\infty_t$ is {\it periodic} if $t = d^pt$ for some $p \geq 1$, or {\it rational} if $d^nt$ is periodic for some $n \geq 0$. It is easy to see that $\cR^\infty_t$ is rational if and only if $t\in \bbQ/\bbZ$. The accumulation set of $\cR^\infty_t$ is denoted $\omega(\cR^\infty_t)$. Note that $\omega(\cR^\infty_t) \subset \hJ_F$. If $\omega(\cR^\infty_t) = \{x\}$, then we say that $\cR^\infty_t$ {\it lands at $x$}.

\begin{prop}
Every periodic external ray of $F$ lands at a repelling or parabolic periodic point in $\hJ_F$. Conversely, every repelling or parabolic periodic point in $\hJ_F$ is the landing point of a periodic external ray.
\end{prop}

\begin{proof}
Let $f$ be a polynomial of degree $d$ obtained from $F$ via the Douady-Ghys surgery. Then there exists a quasiconformal map $\eta$ that maps $(\hA^\infty_F, \infty)$ conformally onto $(A^\infty_f, \infty)$ (see \secref{sec:blaschke}). Under $\eta$, external rays for $F$ maps to external rays for $f$, and $\hJ_F = \partial \hA^\infty_F$ maps to $J_f = \partial A^\infty_f$. The claim now follows from the corresponding result for polynomials (see e.g. \cite{Mi2}).
\end{proof}

By symmetry, $0$ is a fixed critical point. The attracting basin $A^0_F$, the immediate basin $\hE^0_F$, an {\it internal ray} $\cR^0_{-t}$ with {\it internal angle} $-t \in \bbR/\bbZ$, and an equipotential $\cQ_{1/l}$ at level $1/l \in (0, 1)$ are reflections about $\partial \bbD$ of $A^\infty_F$, $\hA^\infty_F$, $\cR^\infty_t$, and $\cQ_l$ respectively.

An {\it external bubble} $B$ of generation $\gen(B) \geq 0$ is defined inductively as follows. The unique external bubble of generation $0$ is $\bbD$. Let $2m+1$ be the degree of the critical point $c_0$. Then there are $m$ connected components of $F^{-1}(\bbD) \cap (\bbC \setminus \bbD)$ whose boundaries have a common intersection point at $c_0$. These components are external bubbles of generation $1$. Let $B_1$ be any external bubble of generation $1$. For $k \geq 1$, a connected component $B_k$ of $F^{-k+1}(B_1)$ is an external bubble of generation $k$ if it is not contained in a bubble of smaller generation. A {\it root} of $B_k$ is a point in $F^{-k+1}(c_0) \cap \partial B_k$.

Let $\{B_i\}_{i=0}^\infty$ be a sequence of external bubbles such that
\begin{itemize}
\item $B_0 = \bbD$; and
\item$\partial B_{i-i} \cap \partial B_i = \{x_i\}$ is the root of $B_i$ for $i \geq 1$.
\end{itemize}
The union
$$
\cR^B = \bigcup_{i=0}^\infty \partial B_i \subset \bbC \setminus \bbD
$$
is an {\it external bubble ray}. See \figref{fig:bubble ray}. The point $x_1 \in \partial \bbD$ is called the {\it root} of $\cR^B$. The accumulation set $\omega(\cR^B)$ of $\cR^B$ is defined as the accumulation set of the sequence $\{B_i\}_{i=0}^\infty$. Note that $\omega(\cR^B) \subset \hJ_F$. If $\omega(\cR^B) = \{x_\infty\}$, then $x_\infty$ is called the {\it landing point} of $\cR^B$.

\begin{figure}[h]
\centering
\includegraphics[scale=0.35]{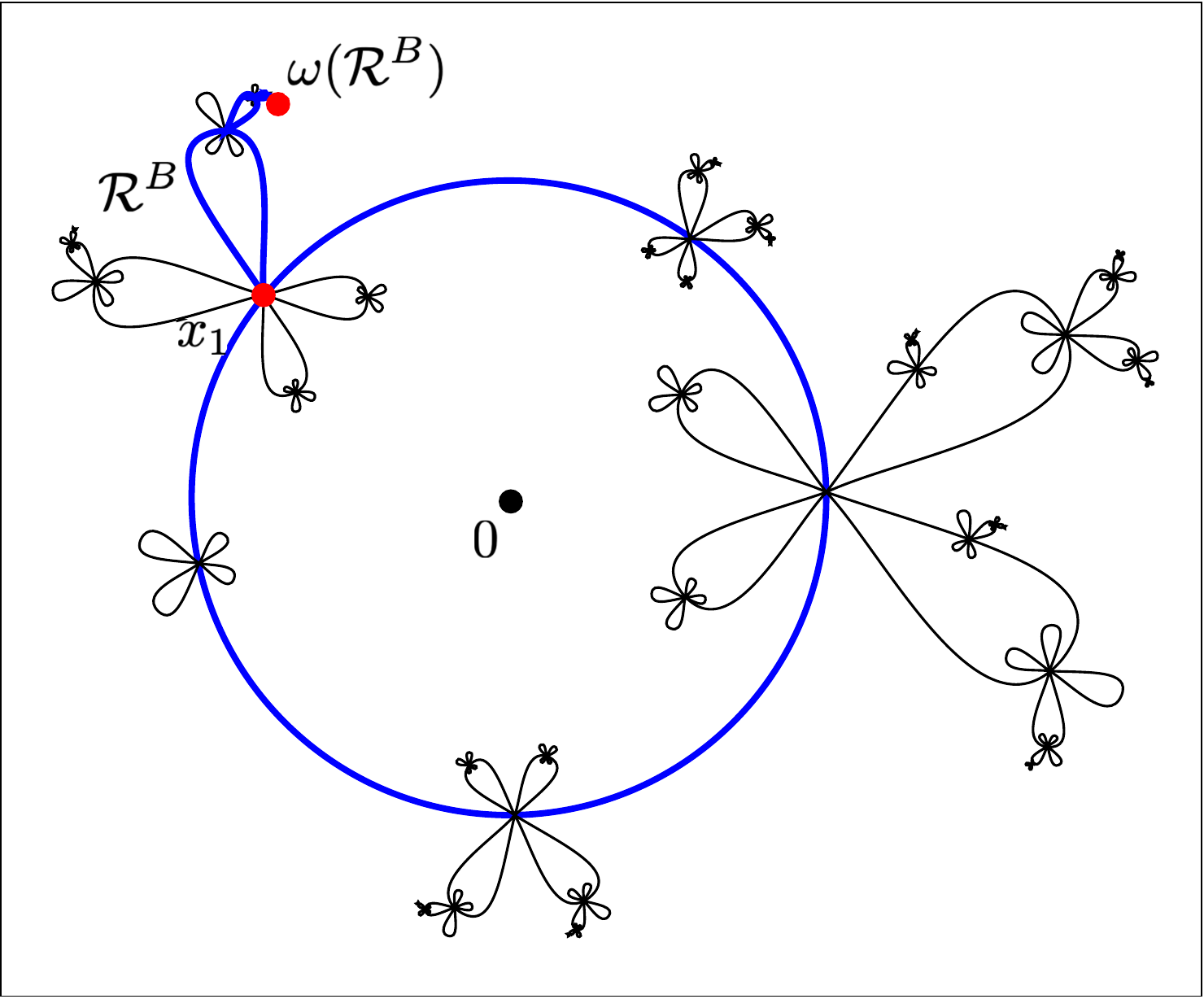}
\caption{An external bubble ray $\cR^B$, its root $x_1$ and limit set $\omega(\cR^B)$.}
\label{fig:bubble ray}
\end{figure}

Observe that the image of an external bubble ray is also an external bubble ray. An external bubble ray $\cR^B$ is {\it periodic} if $F^p(\cR^B) = \cR^B$ for some $p \geq 1$, or {\it rational} if $F^n(\cR^B)$ is periodic for some $n \geq 0$. Note that all fixed external bubble rays are rooted at $c_0$. An external bubble ray $\cR^B$ is said to be {\it $d$-adic of generation $k\geq 0$} if $F^k(\cR^B)$ is a fixed external bubble ray, and $k$ is the smallest number for which this is true.

The {\it impression $\Imp(t)$ at the external angle $t \in \bbR/\bbZ$} is the set of all points $z \in \hJ_F$ for which there exist sequences $t_n \in \bbR/\bbZ$ and $z_n \in \cR^\infty_{t_n}$ such that $t_n \to t$ and $z_n \to z$ as $n \to \infty$. Observe that
$$
F(\Imp(t)) = \Imp(dt).
$$

\begin{prop}[Rational bubble rays land]
Every $p$-periodic external bubble ray $\cR^B$ lands at a repelling or parabolic $p$-periodic point $x_\infty \in \hJ_F$.
\end{prop}

\begin{proof}
We assume for simplicity that $\cR^B$ is fixed. Let $\{B_i\}_{i=0}^\infty$ be the sequence of external bubbles such that
$$
\cR^B = \bigcup_{i=0}^\infty \partial B_i.
$$
Denote the root of $B_i$ by $x_i$. Since $\cR^B$ is fixed, we have $F(B_{i+1}) = B_i$ and $F(x_{i+1}) = x_i$.

For $x \in \hJ_F$, let $A(x) \subset \bbR/\bbZ$ be the set of all external angles $t$ such that $\Imp(t)$ contains $x$. Then we have $A(F(x)) = dA(x)$. In particular, if $A_i := A(x_i)$, then $dA_i = A_{i-1}$.

Let
$$
X_i := \bigcup_{t \in A_i}\left(\overline{\cR^\infty_t} \cup \Imp(t)\right),
$$
and let $U_i$ be the connected component of $\bbC \setminus X_i$ containing $\bbD$. Then $U_i \Subset U_{i+1}$. There exist $t^l_i, t^r_i \in A_i$ such that
$$
\partial U_i \subset \overline{\cR^\infty_{t^l_i}} \cup \Imp(t^l_i) \cup \overline{\cR^\infty_{t^r_i}} \cup \Imp(t^r_i) \ni x_i,
$$
Denote $\gamma_i := [t^l_i, t^r_i] \subset \bbR/\bbZ$. Then $\gamma_i \Supset \gamma_{i+1}$. Let
$$
\gamma_\infty = [t^l_\infty, t^r_\infty] := \bigcap_{i=0}^\infty \gamma_i.
$$
Then $dt^{l/r}_\infty = t^{l/r}_\infty$. Hence, $\cR^\infty_{t^{l/r}_\infty}$ must land at some repelling or parabolic fixed point $x^{l/r}_\infty \in \omega(\cR^B)$.

If $x^r_\infty$ is a repelling fixed point, let $V$ be a sufficiently small neighborhood of $x^r_\infty$. Otherwise, let $V$ be the repelling petal at $x^r_\infty$ which intersect $\cR^\infty_{t^r_\infty}$. Let $g : V \to g(V) \subset V$ be the local inverse branch of $F$. Clearly, $V$ contains a neighborhood of a point $x \in \omega(\cR^B)$. Hence, there exists $y \in B_{i_0} \cap V$ for some $i_0$ sufficiently large. Observe that $g^n(y) \in B_{i_0 +n}$ and $g^n(y) \to x^r_\infty$ as $n \to \infty$.

We claim that $\diam(B_i) \to 0$ as $i \to \infty$. The result then follows from the above observation.

Let
$$
Z_i := U_{i-1} \cup (\bbC \setminus U_{i+2}) \cup \cQ_{2^{1/d^i}},
$$
and let $W_i$ be the connected component of $\bbC\setminus \overline{Z_i}$ compactly containing $B_i$. Observe that for all $i$ sufficiently large, $W_i$ does not contain a critical point of $F$. Hence, $F$ maps $W_i$ conformally onto $W_{i-1}$. The fact that $\diam(B_i)$ goes to zero now follows from Koebe distortion theorem.
\end{proof}

An {\it internal bubble} $\chB \subset \bbD$ of {\it generation} $k$ and an {\it internal bubble ray} $\chR^B \subset \overline{\bbD}$ are the reflections about $\partial \bbD$ of an external bubble $B \subset \bbC \setminus \overline{\bbD}$ of generation $k$ and an external bubble ray $\cR^B \subset \bbC \setminus \bbD$ respectively.

Let $\bfR_0$ be the union of all external and internal bubble rays of generation $0$, all landing points of these bubble rays, and all external and internal rays that also land at these points. Define the {\it initial puzzle partition} as
\begin{equation}\label{eq:0 puzzle}
\cZ_0 := \bfR_0 \cup \cQ_2 \cup \cQ_{1/2}.
\end{equation}
See \figref{fig:puzzle partition}. The {\it puzzle partition of depth} $n \geq 0$ is given by
$$
\cZ_n := F^{-n}(\cF_0).
$$
Denote
\begin{equation}\label{eq:equipotential}
\cQ_+^n := \cQ_{\sqrt[d^n]{2}}.
\matsp{and}
\cQ_-^n := \cQ_{\sqrt[d^n]{1/2}}.
\end{equation}
Then $\cQ_\pm^n \subset \cZ_n$.

\begin{figure}[h]
\centering
\includegraphics[scale=0.35]{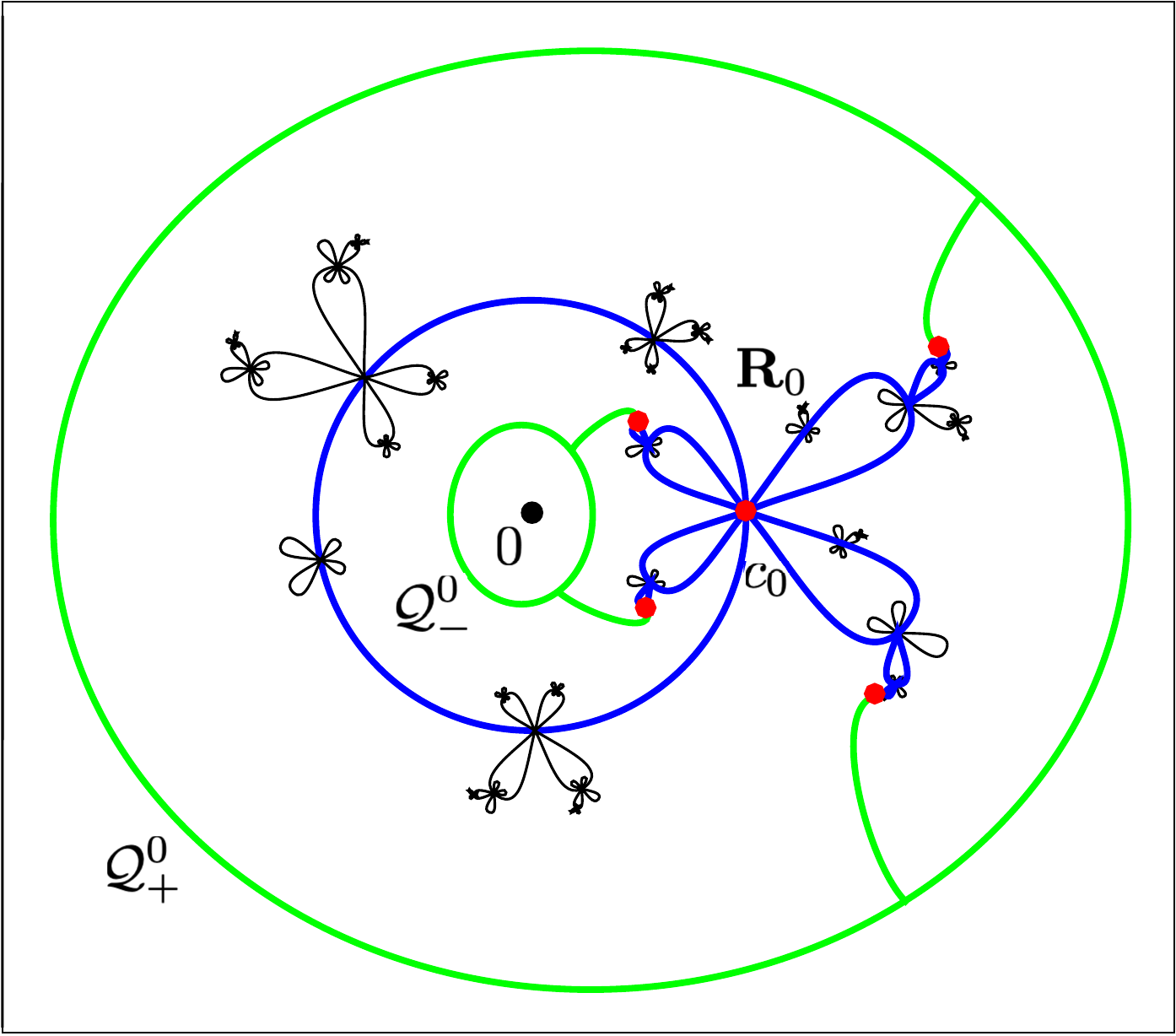}
\caption{The initial puzzle partition $\cZ_0$.}
\label{fig:puzzle partition}
\end{figure}

A connected component of $\bbC \setminus \cZ_n$ is called a {\it puzzle piece of depth $n$}. For $s \in \bbR/\bbZ$, the {\it puzzle neighborhood $P^n(s)$ at angle $s$ of depth $n$} is defined as the interior of the union of closures of puzzles pieces of depth $n$ that contain $\xi_s \in \partial \bbD$ in their boundaries. Define the {\it fiber (of height $0$) at angle $s$} as
$$
X_s := \bigcap_{n = 0}^\infty \overline{P^n(s)}.
$$
Since $h$ conjugates $F$ to a rigid irrational rotation on $\partial \bbD$, the inverse orbit of $1$ is dense in $\partial \bbD$. Thus
$$
X_s \cap \partial \bbD = \{\xi_s\}.
$$
Observe that we have
$$
F(X_s) = X_{s + \rho}.
$$
If $X_s$ contains a critical point, it is referred to as a {\it critical fiber}. In this case, $s$ is called a {\it critical angle}. Denote the set of all critical angles by
$$
\Ang_{\crit} := \{s \in \bbR/\bbZ \; | \; X_s \text{ is a critical fiber}\}.
$$

The {\it puzzle neighborhood of $\pbbD$ of depth $n$} is defined as
\begin{equation}\label{eq:puzzle nbh}
\bfP^n := \bigcup_{s \in \bbR/\bbZ} P^n(s).
\end{equation}

\begin{prop}\label{puzzle neigh inv}
For $n \geq 1$, we have $F(\bfP^n) = \bfP^{n-1}$.
\end{prop}

\begin{proof}
Let $P^n$ be a puzzle piece of depth $n$ such that for some $s \in \bbR/\bbZ$, we have $\xi_s \in \partial P^n$. Then $F(P^n)$ is a puzzle piece of depth $n-1$, and $\xi_{s+\rho} \in \partial F(P^n)$. Hence, $F(\bfP^n) \subset \bfP^{n-1}$.

Conversely, consider a puzzle piece $P^{n-1}$ of depth $n-1$ such that for some $t \in \bbR/\bbZ$, we have $\xi_t \in \partial P^{n-1}$. Let $P^n$ be a component of the preimage of $P^{n-1}$ such that $\xi_{t-\rho} \in \partial P^n$. Then $P^n$ is a puzzle piece of depth $n$. Hence, $P^n \subset \bfP^n$.
\end{proof}

Observe that
$$
\bigcap_{n = 0}^\infty \bfP^n = \bigcup_{s \in \bbR/\bbZ} X_s.
$$
A point $x \in J_F$ is said to be {\it at height $0$} if $x \in X_s$ for some $s \in \bbR/\bbZ$. If $x$ is not at height $0$, then there exists $n \geq 0$ such that $x$ is not contained in $\bfP^n$. In particular, there exists $L \geq 0$ such that the only critical points contained in $\bfP^L$ are at height $0$.

\begin{prop}\label{fiber landing}
Let $\cR^\infty_t \subset \hA^\infty_F$ be an external ray. Suppose that $\omega(\cR^\infty_t)$ nontrivially intersects a fiber $X_s$. Then $\omega(\cR^\infty_t) \subset X_s$. Consequently, if $X_s = \{\xi_s\}$, then $\cR^\infty_t$ lands at $\xi_s$.
\end{prop}

\begin{proof}
Clearly $\omega(\cR^\infty_t)$ cannot intersect two disjoint puzzle neighborhoods. The claim immediately follows.
\end{proof}

Our main motivation for introducing the puzzle partition is the following result.

\begin{prop}[Triviality of fibers implies local connectivity]\label{fiber to lc}
The Julia set $J_F$ and the modified Julia set $\hJ_F$ are locally connected at every point in $\partial \bbD$ if $X_s = \{\xi_s\}$ for all $s \in \bbR /\bbZ$.
\end{prop}

\begin{proof}
To show local connectivity of $J_F$ at $\xi_s \in \pbbD$, one must show that there are arbitrarily small connected open neighborhoods of $\xi_s$ in $J_F$. Unfortunately, if $P^n(s)$ is a puzzle neighborhood of $\xi_s$, then $P^n(s) \cap J_F$ is not connected. Hence, we must make the following modification to our construction of puzzles.

By \propref{fiber landing}, all external and internal rays that accumulate on $\xi_s \in \pbbD$ must land at $\xi_s$. Let $\tilde\bfR_0$ be the union of all external and internal rays that land at the critical point $c_0 = \xi_0 \in \pbbD$. Recall that $c_0$ is the root of all bubble rays of generation $0$. Define the initial modified puzzle partition as
$$
\ticZ_0 := \tilde\bfR_0 \cup \{c_0\} \cup \cQ_2 \cup \cQ_{1/2}.
$$
Proceeding inductively, define the modified puzzle partition of depth $n \geq 1$ by
$$
\ticZ_n := \tilde\bfR_0 \cup \{c_0\} \cup F^{-1}(\ticZ_{n-1}).
$$
Compare with the definition of (unmodified) puzzle partitions \eqref{eq:0 puzzle}. The connected component of $\bbC \setminus \ticZ_n$ containing $\xi_s\in \pbbD$ for $s \in (\bbR/\bbZ)\setminus \{0\}$ is called a modified puzzle neighborhood $\tiP^n(s)$ of depth $n$. It is easy to see that $\tiP^n(s) \cap J_F$ is connected and open in $J_F$. Define the modified fiber at angle $s$ as
$$
\tiX_s := \bigcap_{n = 0}^\infty \tiP^n(s).
$$
Clearly, $\tiX_s$ cannot intersect two disjoint unmodified puzzle neighborhoods. Thus, $\tiX_s \subset X_s = \{\xi_s\}$. Hence, $J_F$ is locally connected at $\xi_s$.

The proof of local connectivity of $\hJ_F$ at $\xi_s$ is identical.
\end{proof}

\section{Puzzle Disks}\label{sec:disks}

As the construction of puzzle neighborhoods in \secref{sec:puzzle} involves taking rather arbitrary unions of puzzle pieces, we have no reason to expect that they have nice transformation properties under iteration by $F$. In this section, we define new dynamically meaningful neighborhoods called {\it puzzle disks} that are much better integrated into the rotational combinatorial structure of $F$ on $\pbbD$.

\subsection{Combinatorics on the circle}

Recall that there is a quasisymmetric map $h : (\pbbD, 1) \to (\pbbD, 1)$ such that for $g := F|_{\pbbD}$, we have
$$
h \circ g \circ h^{-1}(z) = e^{2\pi \rho i} z
\matsp{for}
z\in \pbbD.
$$
Let $\xi_s := h^{-1}(e^{2\pi s i})$ for $s \in \bbR/\bbZ$. We assume that $g$ has a critical point at $c_0 := \xi_0 = 1$. Denote $c_k := g^k(c_0) = \xi_{k\rho}$ for $k \in \bbZ$. Without loss of generality, we may assume that $c_k$ is not a critical point for $k \geq 1$. Note that there exists $l_0 \geq 1$ such that for $l \geq l_0$, the fiber $X_{l\rho} \ni c_l$ is noncritical.

\begin{notn}
For $a, b \in \pbbD$ such that $a \neq \pm b$, let $(a,b)_\pbbD \subset \pbbD$ denote the unique open arc of arclength less than $\pi$ with endpoints $a$ and $b$. The notations $[a, b)_\pbbD$, $(a, b]_\pbbD$ and $[a,b]_\pbbD$ are self-explanatory. An (open) {\it combinatorial arc} is an arc in $\pbbD$ of the form
$$
(n,m)_c := (c_n, c_m)_\pbbD
$$
for some $n, m \in \pbbD$.
\end{notn}

For $n \geq 1$, let $a_n$ be the $n$th coefficient in the continued fraction expansion of $\rho$. Since $\rho$ is of bounded type, there exists a uniform bound $\tau \geq 1$ such that $a_n \leq \tau$. Let $q_n$ be the $n$th closest return time, and define
$$
I^\pm_n := (0, \pm q_n)_c.
$$
Observe that
$$
J^\pm_n := I^\pm_n \cup \{c_0\} \cup I^\pm_{n+1} = (\pm q_n, \pm q_{n+1})_c
$$
is an open neighborhood of $c_0$ in $\partial \bbD$. Moreover,
$$
\cI^\pm_n := \{g^{\pm i}(I^\pm_n) \, | \, 0 \leq i < q_{n+1}\} \cup  \{g^{\pm i}(I^\pm_{n+1}) \, | \, 0 \leq i < q_n\}
$$
is the $n$th dynamic partition of $\partial \bbD$ constructed in \eqref{eq:dyn part} with $g^{\pm 1}$ as the circle homeomorphism and $c_0$ as the initial point.

\begin{lem}\label{pullback interval}
Let $n, m \geq 0$ and $0 \leq k \leq q_{n+m}$. If $J$ is a subarc of $J^-_n$, then the intersection multiplicity of $\cJ := \{g^{-i}(J)\}_{i=0}^k$ is uniformly bounded by a constant depending only on $m$.
\end{lem}

\begin{proof}
It suffices to prove the result for $J = J^-_n$. Consider a maximal subset of arcs in $\cJ$ whose interiors have a nonempty intersection. Without loss of generality, we may assume that one of these arcs is $J$. Let $J' := J^-_{n+m-1}$. Observe that if $g^{-i}(J') \cap J^-_{n-2} = \varnothing$, then $g^{-i}(J) \cap J = \varnothing$. Since $\rho$ is of bounded type, $\#\{0 \leq i \leq q_{n+m} \; | \; g^{-i}(J') \subset J^-_{n-2}\}$ has a uniform bound depending only on $m$.
\end{proof}

\begin{notn}\label{disk iterate}
For $n \geq 0$, denote
$$
r_n := q_n + q_{n+1}
\matsp{and}
\bfr_n := \sum_{i=1}^n r_i.
$$
\end{notn}

\begin{lem}\label{exp time growth}
For $n \geq 3$, we have
\begin{enumerate}[i)]
\item $q_n \geq r_{n-2}$, and equality holds if and only if $a_{n-1} = 1$;
\item $q_{n+1} \geq r_{n-2} + r_{n-3}$, and equality holds if and only if $a_n = a_{n-1} = a_{n-2} = 1$;
\item $\bfr_{n-2} \geq \bfr_n - r_{n+1}$, and equality holds if and only if $a_{n+1} = a_n = 1$; and
\item $r_n > \bfr_{n-2}$.
\end{enumerate}
\end{lem}

\begin{proof}
The first, second and third claims are obvious. For the fourth claim, assume that $r_k > \bfr_{k-2}$ for $k < n$. Then by the first claim, we have $q_n \geq r_{n-2} > \bfr_{n-4}$. The result follows from the second claim.
\end{proof}

\begin{lem}\label{nested arcs}
For $n > 2$, we have
$$
J^+_n \Subset J^-_{n-1} \subset J^-_{n-2}
\matsp{and}
J^-_n = g^{-r_n}(J^+_n) \Subset g^{-r_n}(J^-_{n-1}) \Subset J^-_{n-2}.
$$
\end{lem}

\begin{proof}
We show that
$$
g^{-r_n}(J^-_{n-1}) \Subset J^-_{n-2}.
$$
The other inclusions are obvious.

The arc $g^{-r_n}(J^-_{n-1})$ can be decomposed into three subarcs
$$
g^{-r_n}(J^-_{n-1}) = (-q_n-r_n, -q_n]_c \cup (-q_n, -q_{n+1})_c \cup [-q_{n+1}, -q_{n-1} - r_n)_c.
$$
Consider the arc
$$
J^-_{n-2} = (-q_{n-2}, -q_{n-1})_c \Supset J^-_n = (-q_n, -q_{n+1})_c.
$$
Note
$$
g^{q_n}((-q_{n-2}, -q_n]_c) = (a_{n-1}q_{n-1}, 0]_c \supset (q_{n-1}, 0]_c \supset [r_{n-1}, 0]_c \supset [-r_n, 0]_c.
$$
Hence,
$$
(-q_{n-2}, -q_n]_c \supset [-q_n-r_n, -q_n]_c.
$$
\end{proof}

\subsection{Dividers and puzzle silhouettes}

For $0 \leq n \leq k$, let $\cR^B$ be a dyadic external bubble ray rooted at $c_{-n}$ of generation $k$. Denote its landing point by $x$, and let $\cR^\infty_t$ be an external ray that lands at $x$. Let $\check \cR^B$, $y$ and $\cR^0_{-t}$ be the reflections of $\cR^B$, $x$ and $\cR^\infty_t$ respectively. The set
$$
\cV := ((\cR^B \cup \check \cR^B) \setminus \pbbD) \cup \{c_{-n}, x, y\} \cup \cR^\infty_t \cup \cR^0_{-t}
$$
is called a {\it divider of generation $k$ rooted at $c_{-n}$}. Let $\bfV_n^k$ be the union of all dividers of generation at most $k$ rooted at $c_{-n}$. When convenient, we will abuse notation and write $\cV \in \bfV_n^k$.

Let $I = (-n, -m)_c$ for some $n, m \geq 0$, and let $k \geq \max\{n, m\}$. Let
$$
Z_I^k := \bfV_n^k \cup \bfV_m^k \cup \cQ_+^k \cup \cQ_-^k \subset \cZ_k,
$$
where $\cQ_\pm^k$ are equipotential curves (see \eqref{eq:equipotential}), and $\cZ_k$ is the puzzle partition of depth $k$. A {\it puzzle silhouette $S^k_I$ of $I$ of depth $k$} is the connected component of $\bbC \setminus Z_I^k$ that contains $I$. It is easy to see that
$$
S^k_I \cap \partial \bbD = I.
$$
Moreover, $S^k_I$ is bounded between two dividers $\cV_-(S^k_I) \in \bfV_n^k$ and $\cV_+(S^k_I) \in \bfV_m^k$ which we refer to as the {\it bounding dividers of $S^k_I$}.

\begin{prop}\label{silhouette in puzzle}
Let $l_0\geq 1$ be a number such that for $l \geq l_0$, the fiber $X_{l\rho} \ni c_l$ is noncritical. Given $L \geq 0$, consider the puzzle neighborhood $P^L$ of $c_{l_0}$. Then there exists $N \geq 1$ such that for $n \geq N$ and $k \geq q_{n+1}-l_0$, we have
$$
S^k_{g^{l_0}(J^-_n)} \subset P^L.
$$
\end{prop}

\begin{proof}
Let $I^L \subset P^L \cap \pbbD$ be the maximal open arc containing $c_{l_0}$ such that for all $\xi_s \in I^L$ and $0\leq l < L$, the fiber $X_{s + l\rho}$ is noncritical. Choose $N \geq 1$ such that $q_{N+1}-l_0 > L$, and $g^{l_0}(J^-_N) \subset I^L$. For $n \geq N$, we have $I := g^{l_0}(J^-_n) \subset g^{l_0}(J^-_N)$. Given $k \geq q_{n+1}-l_0$, let $\cV_\pm(S^k_I)$ be the bounding dividers of $S^k_I$. Define
$$
\tiZ^k_I := \cV_+(S^k_I) \cup \cV_-(S^k_I) \cup \cQ_+^L \cup \cQ_-^L,
$$
and let $\tiS^k_I$ be the connected component of $\bbC \setminus \tiZ^k_I$ containing $I$. Then it is easy to see that $\tiS^k_I \cap \cZ_L = I$. Thus, $S^k_I \subset \tiS^k_I \subset P^L$.
\end{proof}

\subsection{Construction of puzzle disks}

Let $U \subset \bbC$ be a connected set whose intersection with $\pbbD$ is an arc $I$. For $k \geq 0$, define the {\it $k$th pullback of $U$ along $\partial \bbD$} to be the connected component $V$ of $F^{-k}(U)$ whose intersection with $\pbbD$ is the arc $g^{-k}(I)$.

\begin{notn}
Let $U \subset \bbC$, and let $I \subset \partial \bbD$ be an arc. Denote
$$
U|_I := U \cap ((\bbC \setminus \partial \bbD) \cup I).
$$
\end{notn}

Recall that there exists $L \geq 0$ such that if $c$ is a critical point contained in the puzzle neighborhood $\bfP^L$ of $\pbbD$ of depth $L$, then $c$ is contained in a fiber $X_s$ for some critical angle $s \in \Ang_{\crit} \subset \bbR/\bbZ$. For this value of $L$, let $N\geq 1$ be the number given in \propref{silhouette in puzzle}.

\begin{lem}\label{initial scale}
There exists $n_0 > N$ such that the following holds. Let 
$$
\fJ := g^{l_0}(J_{n_0}^-)
\matsp{and}
\fr := \bfr_{n_0}-l_0,
$$
where $\bfr_{n_0}$ is given in \notref{disk iterate}. Then
\begin{enumerate}[i)]
\item The boundary of the puzzle silhouette $S^\fr_\fJ$ does not intersect the postcritical set of $F$.
\item If $J = g^i(\fJ)$ for some $i \in \bbZ$, then $J$ contains at most one critical angle $s \in \Ang_{\crit}$.
\end{enumerate}
\end{lem}

\begin{proof}
For $i \geq 0$, we have 
$$
F^{-i}((\partial S^\fr_\fJ \cap J_F) \setminus \pbbD) \subset \cZ_{\fr+i} \setminus \cZ_{q_{n_0} - l_0+i}.
$$
The first claim follows. The second claim is an immediate consequence of \corref{hermancor}. 
\end{proof}

We refer to $S^\fr_\fJ$ in \lemref{initial scale} as the {\it initial puzzle silhouette}. For $n \geq n_0$, we define the {\it puzzle disk $D^n$} of {\it scale $n$} as follows.

First, let $D^{n_0}$ be the $l_0$th pullback of $S^\fr_\fJ$ along $\partial \bbD$. Proceeding inductively, suppose $D^{n-1}$ is defined so that
$$
D^{n-1} \cap \partial \bbD = J^-_{n-1}.
$$
By \lemref{nested arcs}, we have
$$
g^{r_n}(J^-_n) = J^+_n \Subset J^-_{n-1}.
$$
For $0 \leq i \leq r_n$, let $W^n_{-i}$ be the $i$th pullback of the slitted domain $W^n_0 := D^{n-1}|_{J^+_n}$ along $\partial \bbD$. Then define $D^n := W^n_{-r_n}$. See \figref{fig:puzzle disks}. The {\it depth} of the puzzle disk $D^n$ of scale $n$ is defined as $\bfr_n$.

\begin{figure}[h]
\centering
\includegraphics[scale=0.35]{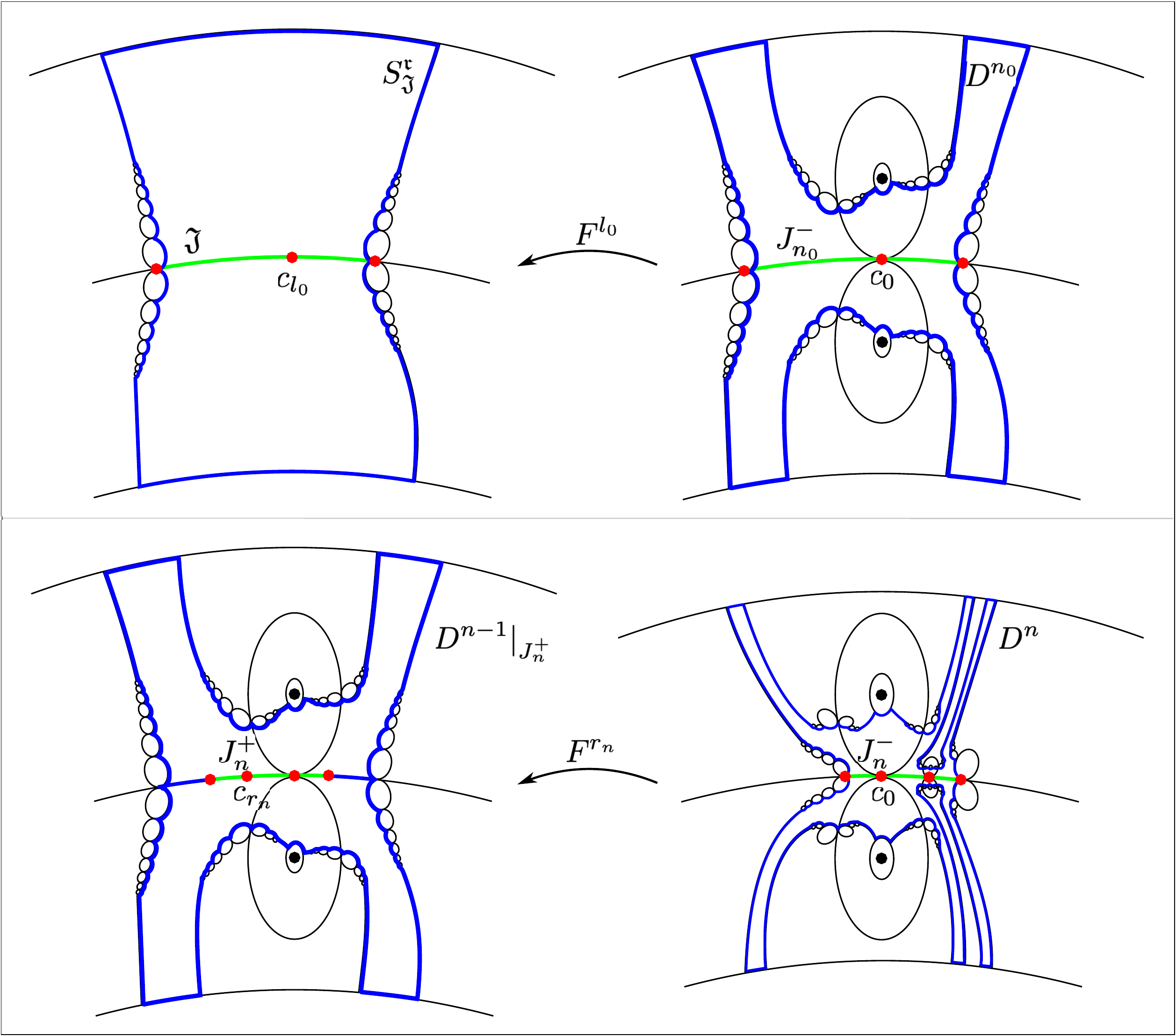}
\caption{Top: the puzzle disk $D^{n_0}$ defined as the $l_0$th pullback of the initial puzzle silhouette $S^\fr_\fJ$ along $\pbbD$. Bottom: the puzzle disk $D^n$ defined as the $r_n$th pullback of the slitted puzzle disk $D^{n-1}|_{J^+_n}$ along $\pbbD$.}
\label{fig:puzzle disks}
\end{figure}

\subsection{Pulling back puzzle disks along the circle}

Let $\fU$ be either the initial puzzle silhouette $S^\fr_\fJ$ or the puzzle disk $D^n$ of scale $n \geq n_0$. Given $K \geq 0$ and an open combinatorial arc $\gamma_0 \subset \fU \cap \pbbD$, consider the $k$th pullback $U_{-k}$ of $U_0 := \fU|_{\gamma_0}$ along $\pbbD$ for $0 \leq k \leq K$. Denote
$$
\gamma_{-k} = (c_{m^-_{-k}}, c_{m^+_{-k}})_\pbbD := g^{-k}(\gamma_0),
$$
$$
\bar\Gamma_{-k} := [e^-_{-k}, e^+_{-k}]_\pbbD := \overline{U_{-k}} \cap \pbbD,
$$
$$
\Gamma_{-k} := (e^-_{-k}, e^+_{-k})_\pbbD
\matsp{,}
\gamma_{-k}^- = (e^-_{-k}, c_{m^-_{-k}}]_\pbbD
\matsp{and}
\gamma_{-k}^+ = [c_{m^+_{-k}}, e^+_{-k})_\pbbD
$$
where $m^\pm_{-k} \in \bbZ$ and $e^\pm_{-k} \in \pbbD$. Then
$$
\Gamma_{-k} = \gamma_{-k}^- \sqcup \gamma_{-k} \sqcup \gamma_{-k}^+.
$$
The set $U_{-k}$ is called a {\it puzzle disk pullback}. The arcs $\gamma_{-k}$ and $\Gamma_{-k}$ are referred to as the {\it base} and the {\it full base of $U_{-k}$} respectively. Let $\fd \geq 0$ be the depth of $\fU$. Then the {\it depth} of $U_{-k}$ is defined to be $d_{-k} := \fd + k$. 

The boundary of $\overline{U_{-k}}$ is a Jordan loop contained in $F^{-d_{-k}}(\bfV_0^0) \cup \cQ_+^{d_{-k}} \cup \cQ_-^{d_{-k}}$, where $\bfV_0^0$ denotes the union of all dividers of generation $0$ rooted at $c_0$, and $\cQ_{\pm}^{d_{-k}}$ are equipotential curves. Let $\ticE_\pm(U_{-k})$ be the connected component of $\partial \overline{U_{-k}} \setminus (\cQ_+^{d_{-k}} \cup \cQ_-^{d_{-k}})$ containing $e^\pm_{-k}$. The following observation is obvious.

\begin{prop}[Transformation of edges]\label{edge}
For $0\leq k < K$, let $\tilde \gamma^\pm_{-k}$ be either
\begin{itemize}
\item $\{e^\pm_{-k}\}$ if $\gamma^\pm_{-k}$ does not contain a critical value; or
\item the maximal closed subarc of $\gamma^\pm_{-k}$ whose endpoints are $c_{M^\pm_{-k}}$ and a critical value $v^\pm_{-k} \in \gamma^\pm_{-k}$.
\end{itemize}
Then
$$
F(\gamma^\pm_{-k-1}) = (\gamma^\pm_{-k} \setminus \tilde \gamma^\pm_{-k}) \cup \{v^\pm_{-k}\}
\matsp{and}
F(\ticE_\pm(U_{-k-1})) = \ticE_\pm(U_{-k}) \cup \tilde\gamma^\pm_{-k}.
$$
Consequently, the following holds.
\begin{enumerate}[i)]
\item There exist unique symmetric external and internal ray $\cR^\infty_{t_\pm}$ and $\cR^0_{-t_\pm}$ that intersect $\ticE_\pm(U_{-k})$.
\item $e^\pm_{-k} = c_{m^\pm_{-k}}$ if $0< m^\pm_0 \leq k$.
\end{enumerate}
\end{prop}

Define the {\it bounding edges} of $U_{-k}$ as
$$
\cE_\pm(U_{-k}) := \ticE_\pm(U_{-k}) \cup \cR^\infty_{t_\pm} \cup \cR^0_{-t_\pm},
$$
where $\cR^\infty_{t_\pm}$ and $\cR^0_{-t_\pm}$ are given in \propref{edge}. The angle $t_\pm$ is referred to as the {\it external angle of $\cE_\pm(U_{-k})$}.

\begin{prop}[Simple connectivity of pullbacks of puzzle disks]
For $0 \leq k \leq K$, the puzzle disk pullback $U_{-k}$ is also simply connected. In particular, for $n \geq n_0$, the puzzle disk $D^n$ is simply connected.
\end{prop}

\begin{proof}
If $\fU$ is simply connected, then certainly $U_0 := \fU|_{\gamma_0}$ is simply connected. Assume that $U_{-k+1}$ is simply connected for $0< k \leq K$. Suppose towards a contradiction that $U_{-k}$ is not simply connected. Then $\partial U_{-k}$ has at least two components $\Delta_{\ext}$ and $\Delta_{\Int}$ such that $\Delta_{\Int}$ is contained in the bounded component of $\bbC \setminus \Delta_{\ext}$. Moreover, $\Delta_{\ext}$ and $\Delta_{\Int}$ both cover $\partial U_{-k+1}$ under $F$. Thus, $\Delta_{\ext}$ and $\Delta_{\Int}$ both contain an arc in $Q_\pm^{d_{-k}}$. This is impossible.
\end{proof}

\begin{prop}\label{fiber in disks}
If $\xi_s \in \gamma_{-k}$, then $X_s \subset U_{-k}$. If instead, $\xi_s \notin \bar\Gamma_{-k}$, then $X_s \cap U_{-k} = \varnothing$. In particular, if $\xi_s \in J^-_n$, then $X_s \subset D^n$, and if $\xi_s \notin \bar J^-_n$, then $X_s \cap D^n = \varnothing$. 
\end{prop}

\begin{proof}
Since $\partial \fU \subset \cZ_{d_0}$, we have $\partial U_{-k} \subset \cZ_{d_{-k}}$. Thus, any puzzle piece of depth $d_{-k}$ or greater must either be contained in $U_{-k}$ or be disjoint from it. The first claim follows. Let $P^i$ be a puzzle piece of depth $i \geq d_{-k}$ such that $P^i \cap U_{-k} = \varnothing$. If $\overline{P^i} \cap \bar\Gamma_{-k}$ is non-empty, then it must consist of either $e^+_{-k}$ or $e^-_{-k}$. The second claim follows.
\end{proof}

\begin{prop}[Degree bound on pullbacks of puzzle disks]\label{crit in disk}
If $k \leq q_{n+m}$ for some $m \geq 0$, then the degree of $F^k|_{U_{-k}}$ is uniformly bounded by a constant depending only on $m$. In particular, the degree of $F^{r_{n+1}}|_{D^{n+1}} : D^{n+1} \to D^n$ is uniformly bounded.
\end{prop}

\begin{proof}
Recall that $L \geq 0$ is chosen so that the puzzle neighborhood $\bfP^L$ of $\pbbD$ only contains critical points of height $0$. By \propref{puzzle neigh inv}, we have $U_{-k} \subset \bfP^L$ for all $0 \leq k \leq K$. The result is now an immediate consequence of \lemref{pullback interval} and \propref{fiber in disks}.
\end{proof}

\begin{prop}\label{disk pull in sil}
Suppose that $\fU = S^\fr_\fJ$, $k \geq r_{n_0+1}$, and $\gamma_{-k} \subset \fJ$. Then $\Gamma_{-k} \subset \fJ$ and $U_{-k} \subset S^\fr_\fJ$.
\end{prop}

\begin{proof}
Suppose towards a contradiction that $\Gamma_{-k}$ is not contained in $\fJ = (-q_{n_0} + l_0, -q_{n_0+1}+l_0)_c$. For concreteness, assume $\Gamma_{-k}$ contains $c_{-q_{n_0+1}+l_0}$ in its interior. Then $\Gamma_{-k+q_{n_0}-l_0}$ contains $c_0$ in its interior. Since $k > q_{n_0}-l_0$, this contradicts \propref{edge}.

Consider the bounding edges $\cE_\pm(S^\fr_\fJ)$ and $\cE_\pm(U_{-k})$ of $S^\fr_\fJ$ and $U_{-k}$ respectively. Additionally, let $s_\pm$ and $t_\pm$ be the external angles of $\cE_\pm(S^\fr_\fJ)$ and $\cE_\pm(U_{-k})$ respectively. Then $\cR^\infty_{s_\pm} \subset \cE_\pm(S^\fr_\fJ)$ and $\cR^\infty_{t_\pm} \subset \cE_\pm(U_{-k})$. Clearly, if $(t_-, t_+) \Subset (s_-, s_+)$, then $U_{-k} \subset S^\fr_\fJ$.

Suppose towards a contradiction that this is not the case. For concreteness, assume that $t_+ > s_+$. Since the immediate attracting basin $\hA^\infty_F$ is connected, we see that
$$
E_+ := \cE_+(S^\fr_\fJ) \cap \cE_+(U_{-k})
$$
is either a Jordan arc (if the endpoint $e^+_{-k}$ of $\Gamma_{-k}$ is also an endpoint of $\fJ$) or an empty set (if $e^+_{-k}$ is contained in the interior of $\fJ$). Since $t_+ > s_+$, the latter case is impossible. Thus, $E_+$ is a Jordan arc with an endpoint at $e^+_{-k}$.

Recall that
$$
F^k(U_{-k}) =  U_0 := (S^\fr_\fJ)|_{\gamma_0},
$$
and we have
$$
\fJ = \Gamma_0 = \gamma^-_0 \sqcup \gamma_0 \sqcup \gamma^+_0.
$$
By \propref{edge}, the set
$$
\tiE_+ := (F^k|_{\cE_+(U_{-k})})^{-1}(\gamma^+_0)
$$
is a Jordan subarc of $\cE_+(U_{-k})$ with an endpoint at $e^+_{-k}$. Moreover, we have
$$
F^k(\cE_+(U_{-k}) \setminus \tiE_+) = \cE_+(S^\fr_\fJ).
$$
It follows that
$$
(\cE_+(U_{-k}) \setminus \tiE_+) \cap \cE_+(S^\fr_\fJ) = \varnothing.
$$
Thus, $E_+ \subset \tiE_+$.

Let $\cR^B_+$ be the external bubble ray of generation at most $\fr$ that contains $\cE_+(S^\fr_\fJ)$. Then the external ray $\cR^\infty_{s_+}$ lands at the same point as $\cR^B_+$. Let $\{B_i\}_{i=0}^\infty$ be the sequence of external bubbles of increasing generation such that the union of their boundaries forms $\cR^B_+$. Denote the root of $B_i$ by $x_i \in \partial B_i$. Since $t_+ > s_+$, there exists $j \geq 1$ such that $B_j \cap U_{-k} \neq \varnothing$. Let $j$ be the smallest such number. Then it is not hard to see that $x_i$ is an endpoint of $E_+$.

If $\gen(B_j) < k$, then $\Gamma_{-k+\gen(B_j)}$ contains $c_0 = F^{\gen(B_j)}(x_j)$ in its interior, which contradicts \propref{edge}. Thus, $F^k(\cR^B_+)$ is an external bubble ray rooted at
$$
c_{-j'} := F^k(x_j) \in \gamma^+_0 \subset \fJ = (-q_{n_0} + l_0, -q_{n_0+1}+l_0)_c.
$$
By the combinatorics of first return moments, it follows that
$$
j' \geq r_{n_0} - l_0 > q_{n_0+1}.
$$
However, we have the following bound on the generation of $F^k(\cR^B_+)$:
$$
\fr - k \leq \bfr_{n_0}-l_0 - r_{n_0+1}  \leq \bfr_{n_0-2} \leq q_{n_0},
$$
where the last two inequalities are given by \lemref{exp time growth} iii) and i) respectively. Thus $j' > \fr-k$, which is a contradiction.
\end{proof}

\begin{prop}[Pulling back into puzzle disks]\label{disk pull in}
Suppose that $\fU = D^n$, $k \geq r_n$, and $\gamma_{-k} \subset J^-_n$. Then $\Gamma_{-k} \subset J^-_n$ and $U_{-k} \subset D^n$. In particular, we have $D^{n+1} \subset D^n$.
\end{prop}

\begin{proof}
First, consider the case $n = n_0$. We have $g^{l_0}(\gamma_0), \gamma_{-k+l_0} \subset \fJ$. It is easy to see that $U_{-k+l_0}$ is equal to the $k$th pullback of $(S^\fr_\fJ)|_{g^{l_0}(\gamma_0)}$. By \propref{disk pull in sil}, we have $\Gamma_{-k+l_0} \subset \fJ$ and $U_{-k+l_0} \subset S^\fr_\fJ$. The result immediately follows.

Proceeding by induction, assume that the statement is true for $n-1 \geq n_0$. Suppose towards a contradiction that
$$
\Delta:= U_{-k} \cap \partial D^n \neq \varnothing.
$$
Recall that we have
$$
F^{r_n}(D^n) = D^{n-1}|_{J^+_n}.
$$
Since
$$
F^{r_n}(U_{-k}) \cap \pbbD = g^{r_n}(\gamma_{-k}) = \gamma_{-k+r_n},
$$
we have
$$
F^{r_n}(\Delta) \cap \pbbD =\varnothing.
$$
Hence,
$$
F^{r_n}(\Delta) \cap D^{n-1} \subset (F^{r_n}(\partial D^n) \cap D^{n-1}) \setminus \pbbD = \varnothing,
$$
since otherwise, $\Delta$ would have a non-trivial intersection with $D^n$ which is impossible. We conclude that $F^{r_n}(\Delta) \subset \partial D^{n-1}$.

Let $U^{n-1}_{-k}$ be the $k$th pullback of $D^{n-1}|_{g^{r_n}(\gamma_0)}$ along $\partial \bbD$. Then
$$
U^{n-1}_{-k} = F^{r_n}(U_{-k}) = U_{-k+r_n}.
$$
Since ${g^{r_n}(\gamma_{-k})} \subset g^{r_n}(J^-_n) \subset J^-_{n-1}$, we have $U^{n-1}_{-k} \subset D^{n-1}$ by the induction hypothesis. However,
$$
F^{r_n}(\Delta) = U^{n-1}_{-k} \cap \partial D^{n-1} \neq \varnothing.
$$
This is a contradiction.
\end{proof}

\begin{prop}[Pulling back into puzzle disks of deeper scale]\label{disk pull in small}
Suppose that $\fU = D^n$, $k \geq \bfr_{n+i} - \bfr_{n-1}$ with $i \geq 1$, and $\gamma_{-k} \subset J^-_{n+i}$. Then $\Gamma_{-k} \subset J^-_{n+i}$ and $U_{-k} \subset D^{n+i}$.
\end{prop}

\begin{proof}
Denote $R := \bfr_{n+i} - \bfr_n$. Then $k - R \geq r_n$, and 
$$
F^R(D^{n+i}) = D^n|_{g^R(J^-_{n+i})}.
$$
Moreover,  $\gamma_{-k+R} \subset g^R(J^-_{n+i}) \subset J^-_n$. By \propref{disk pull in}, we have $U_{-k+R} \subset D^n|_{g^R(J^-_{n+i})}$. Hence, $U_{-k}$ is contained in the $R$th pull back of $D^n|_{g^R(J^-_{n+i})}$ which is equal to $D^{n+i}$.
\end{proof}

\begin{prop}[Puzzle disks are nested]\label{nested disks}
We have $D^{n+2} \Subset D^n$.
\end{prop}

\begin{proof}
Suppose towards a contradiction that
$$
\Delta := \partial D^{n+2} \cap \partial D^n \neq \varnothing.
$$
Denote $R := \bfr_n-\bfr_{n_0}$. We have $F^R(D_n) = D^{n_0}|_{g^R(J^-_n)}$ and $F^{R+l_0}(\Delta) \subset \partial (S^\fr_\fJ)|_{g^{R+l_0}(J^-_n)}$. Note that $r_{n+2} > R+l_0$ by \lemref{exp time growth} iv).

We have $F^{r_{n+2}}(D^{n+2}) = D^{n+1}|_{J^+_{n+2}}$, where
$$
J^+_{n+2} = g^{r_{n+2}}(J^-_{n+2}) \Subset J^-_{n+1}. 
$$
Consider
$$
\hJ^0_{n+2} := g^{-r_{n+2}}(J^-_{n+1}) = (-q_{n+2}-r_{n+2}, -q_{n+1}-r_{n+2})_c \Supset J^-_{n+2}.
$$
Then $\hJ^0_{n+2} \Subset J^-_n$. Thus, we have
$$
F^{R+l_0}(\partial D^{n+2}) \cap \pbbD \subset \hJ^1_{n+2}: =g^{R+l_0}(\hJ^0_{n+2}) \Subset g^{R+l_0}(J^-_n) \subset \fJ.
$$
It follows that $F^{R+l_0}(\Delta) \subset \partial S^\fr_\fJ$. This contradicts the fact that the immediate attracting basin $\hA^\infty_F$ is connected.
\end{proof}

\section{Conformal Geometry Near the Circle}\label{sec:geometry}

In this section, we use {\it a priori} bounds for analytic circle maps (discussed in \secref{sec:a priori}) to control the conformal geometry of pullbacks of puzzle disks (constructed in \secref{sec:disks}) near $\pbbD$.

\subsection{Basic properties of extremal Lengths}

Given a path family $\Gamma$ in $\bbC$, denote its extremal length by $\cL(\Gamma)$. The {\it extremal width} of $\Gamma$ is defined as $\cW(\Gamma) := \cL(\Gamma)^{-1}$. Below we briefly review some basic properties of extremal lengths and widths. See e.g. \cite{Ly} for the proofs of these results.

Let $\Gamma_0$, $\Gamma_1$ and $\Gamma_2$ be path families in $\bbC$. We say that $\Gamma_0$ {\it overflows} $\Gamma_1$ if each path in $\Gamma_0$ contains a path in $\Gamma_1$. We say that $\Gamma_0$ {\it disjointly overflows} $\Gamma_1$ and $\Gamma_2$ if any path $\gamma_0 \in \Gamma_0$ contains a pair of disjoint paths $\gamma_1 \in \Gamma_1$ and $\gamma_2 \in \Gamma_2$.

\begin{lem}\label{overflow}
If $\Gamma_0$ overflows $\Gamma_1$, then $\cL(\Gamma_0) \geq \cL(\Gamma_1)$. If $\Gamma_0$ disjointly overflows $\Gamma_1$ and $\Gamma_2$, then $\cL(\Gamma_0) \geq \cL(\Gamma_1) + \cL(\Gamma_2)$.
\end{lem}

We say that $\Gamma_1$ and $\Gamma_2$ are {\it disjoint} if they are in disjoint measurable subsets of $\bbC$.

\begin{lem}\label{union width}
If $\Gamma_0 = \Gamma_1 \cup \Gamma_2$, then $\cW(\Gamma_0) \leq \cW(\Gamma_1) + \cW(\Gamma_2)$. Equality holds if $\Gamma_1$ and $\Gamma_2$ are disjoint.
\end{lem}

\begin{notn}\label{not:path family}
Let $Q \subset \bbC$ be a domain, and let $I, J \subset \overline{Q}$. Denote by $\Gamma_Q(I, J)$ the path family in $Q$ consisting of paths with one endpoint in $I$ and the other in $J$.
\end{notn}

Let $U, V$ be domains such that $U \Subset V$. The modulus of the annulus $A := V \setminus \overline{U}$ is given by
$$
\mod(A) := \cL(\Gamma_A(\partial U, \partial V)).
$$
We refer to $U$ as the {\it inner component} of $A$. For any set $X \subset U$, we say that $A$ {\it surrounds} $X$. By \lemref{overflow}, if $A \subset A'$, and the inner component of $A$ contains that of $A'$, then
$$
\mod(A) \leq \mod(A').
$$
A sequence of disjoint annuli $\{A^n\}_{n=1}^\infty$ in $\bbC$ are said to be {\it nested} if $A^{n+1}$ is contained in the inner component of $A^n$.

\begin{lem}[Gr\"otzsch inequality]\label{grot}
Let $x \in X \subset \bbC$, and let $\{A^n\}_{n=1}^\infty$ be a sequence of nested annuli surrounding $X$. If
$$
\sum_{n=1}^\infty \mod(A^n) = \infty,
$$
then $X = \{x\}$.
\end{lem}

\begin{lem}\label{dist to mod}
Let $Q \subset \bbC$ be a domain, and let $I, J \subset \overline{Q}$. Suppose for some $\mu > 0$, we have
$$
\dist(I, J) > \mu \diam(I).
$$
Then there exists $C = C(\mu) > 0$ such that $\cL(\Gamma_Q(I,J)) > C$.
\end{lem}

\begin{lem}\label{pullback mod}
Let $U' \Subset U$ and $V' \Subset V$ be a pair of nested topological disks, and let $f : (U, U') \to (V, V')$ be a holomorphic branched covering of respective topological disks. Then
$$
\mod (U \setminus \overline{U'}) \leq \mod(V \setminus \overline{V'}) \leq \deg(f)\mod(U \setminus \overline{U'}).
$$
\end{lem}

\subsection{Extremal lengths between pullbacks of puzzle disks and the circle}

Let $D^n$ be the puzzle disk of scale $n \geq n_0$ constructed in \secref{sec:disks}. Recall that the base of $D^n$ is given by
$$
D^n \cap \pbbD = J^-_n := (-q_n, -q_{n+1})_c.
$$
Moreover, for $n_0 \leq n' \leq n$, we have
$$
F^{\bfr_n - \bfr_{n'}}(D^n) = D^{n'}|_{g^{\bfr_n - \bfr_{n'}}(J^-_n)}.
$$
Lastly, there exists $l_0 \in \bbN$ such that $S^\fr_\fJ = F^{l_0}(D^{n_0})$ is the initial puzzle silhouette.

Given $k \geq 1$ and a combinatorial arc $I \subset J^-_n$, consider the $k$th pullback $U$ of $D^n|_I$ along $\pbbD$. Let $e_\pm \in \pbbD$ such that the full base of $U$ is given by
$$
J := (e_-, e_+)_\pbbD.
$$
Then
$$
\overline{J} = \overline{U} \cap \pbbD.
$$
Let $\cE_\pm(U)$ be the bounding edge of $U$ that contains $e_\pm$.

\begin{notn}
Let $\gamma \subset \pbbD$ be an arc. For $\lambda >0$, let $\gamma[\lambda] \subset \pbbD$ be an arc compactly containing $\gamma$ such that for the two components $\gamma[\lambda]_-$ and $\gamma[\lambda]_+$ of $\gamma[\lambda] \setminus \gamma$, we have 
$$
|\gamma[\lambda]_-| = |\gamma[\lambda]_+| = \lambda|\gamma|.
$$
\end{notn}

For $\lambda >0$, let $J[\lambda]_\pm$ be the component of $J[\lambda]\setminus J$ containing $e_\pm$.

\begin{lem}\label{big lambda}
There exists $\Lambda(n) > 0$ with $\Lambda(n) \to \infty$ as $n \to \infty$ such that for $\lambda < 2\Lambda(n)$, we have
$$
g^{\bfr_n + k - \fr}(J[\lambda]) \subset \fJ \subset \pbbD.
$$
\end{lem}

\begin{proof}
By \thmref{geometry bound}, we have $|J^-_n| \to 0$ as $n \to \infty$. Moreover, $g^{\bfr_n- \bfr_{n_0}}(J^-_n) \subset J^+_{n_0+1} \Subset J^-_{n_0}$. Denote the two components of $J^-_{n_0} \setminus g^{\bfr_n- \bfr_{n_0}}(J^-_n)$ by $\gamma^-_n$ and $\gamma^+_n$. Then
$$
\frac{|g^{\bfr_n- \bfr_{n_0}}(J^-_n)|}{|\gamma^\pm_n|} \to 0
\matsp{as}
n \to \infty.
$$
We conclude
$$
\frac{|g^{-k}(J^-_n)|}{|g^{-\bfr_n+ \bfr_{n_0} -k}(\gamma^\pm_n)|} \to 0
\matsp{as}
n \to \infty
$$
by \corref{hermancor}.
\end{proof}

Let $C \subset \bbC \setminus \{0\}$ be a smooth simple curve. Consider a lift $\hC$ of $C$ via the map $\ixp(z) := e^{2\pi i z}$ such that $\ixp(\hC) = C$. Then we say that the slope of $C$ is bounded absolutely from below by $\mu >0$ if this is true for $\hC$.

\begin{prop}[Bounded geometry of edges near $\pbbD$]\label{a priori geometry}
Suppose $k \leq q_{n+m}$ for some $m \geq 0$. For $1< \lambda < \Lambda(n)$, there exist uniform constants $\eta = \eta(m, \lambda) > 0$ independent of $n$, and $\mu >0$ independent of $n$, $m$ and $\lambda$ such that the following holds. The intersection of $\cE_\pm(U)$ with the $\eta|J|$-neighborhood $N_{\eta|J|}(J[\lambda])$ of $J[\lambda]$ consists of a single piecewise smooth curve $E_\pm$. Moreover, $E_\pm$ is equal to the union of two smooth curves $E^\infty_\pm \subset \bbC \setminus \bbD$ and $E^0_\pm \subset \overline{\bbD}$ that are symmetric about $\pbbD$, share an endpoint at $e_\pm \subset \pbbD$, and have slopes that are bounded absolutely from below by $\mu$.
\end{prop}

\begin{proof}
Denote
$$
R:= \bfr_n - \fr + k.
$$
Then
$$
g^R(J[2\lambda]) \subset \fJ
$$
by \lemref{big lambda}, and
\begin{equation}\label{eq:a priori geometry 1}
F^R(\cE_\pm(U)) = \cE_\pm(S^\fr_\fJ) \cup (\fJ \setminus g^R(J)),
\end{equation}
by \propref{edge}.

Using \lemref{exp time growth} i) and iv), we see that
$$
R < q_{n+4} + q_{n+m} < q_{n+m+4}.
$$
For $i \in \bbZ$, denote
$$
\gamma_i := [-k+iq_{n+m+3}, -k+(i+1)q_{n+m+3}]_c.
$$
\corref{hermancor} implies that there exist uniform constants $M = M(m, \lambda) \in \bbN$ and $\kappa = \kappa(m, \lambda)$ independent of $n$ such that
\begin{equation}\label{eq:a priori geometry 2}
J[\lambda] \subset \bigcup_{|i| < M} \gamma_i
\matsp{and}
|\gamma_i| > \kappa|J|
\matsp{for}
|i|<M.
\end{equation}

Let $\hg :\bbR \to \bbR$ be a lift of $g : \pbbD \to \pbbD$ via the map $\ixp(x) := e^{2\pi i x}$. Additionally, let $\hJ, \hJ[\lambda], \hgamma_i \subset \bbR$ be lifts of $J, J[\lambda], \gamma_i \subset \pbbD$ respectively such that
$$
\hJ \subset \hJ[\lambda] \subset \bigcup_{|i| < M} \hgamma_i.
$$
By \thmref{distortion bound} and \thmref{complex extension}, the map $\hg^R|_{\hgamma_i}$ with $|i| < M$ factors into a composition of power maps and diffeomorphisms such that
\begin{itemize}
\item the length of the composition is uniformly bounded;
\item the degrees of the power maps are uniformly bounded; and
\item the diffeomorphisms have uniform complex extensions.
\end{itemize}
\lemref{initial scale} ii) and \eqref{eq:a priori geometry 2} imply that $\hg^R|_{\hJ[\lambda]}$ also factors in a similar way, except the length of the composition is bounded by some constant $L = L(m, \lambda) \geq 1$, and the diffeormophisms in the composition have $\eta'$-complex extensions for some $\eta' = \eta'(m, \lambda) >0$, where $L$ and $\eta'$ are both independent of $n$. 

Consequently, there exist a uniform constant $\eta = \eta(m, \lambda)>0$ such that for 
$$
W := N_{\eta|J|}(J[\lambda])
$$
we have $\Crit(F^R|_W) = \Crit(g^R|_{J[\lambda]})$. From \lemref{big lambda}, we see that there exists a uniform constant $C>0$ such that
$$
V := N_C(g^R(J[\lambda])) \Subset S^\fr_\fJ.
$$
By decreasing $\eta$ if necessary (but still keeping it larger than some uniform lower bound independent of $n$), we have $F^R(W) \subset V$. Hence, by \eqref{eq:a priori geometry 1}, we have
$$
F^R(W \cap \cE_\pm(U)) \subset V \cap \fJ \subset \pbbD.
$$

Let
$$
X := (F^R|_W)^{-1}(g^R(J[\lambda]))\setminus J[\lambda].
$$
It follows from \propref{ext a priori} that each component of $X$ has slope that is bounded absolutely from below by some uniform constant $\mu >0$ independent of $n$. Moreover, we must have
$$
(W \cap \cE_\pm(U)) \cap (\bbC \setminus \pbbD) \subset X.
$$
The result follows.
\end{proof}

For $x_0 \subset \bbR$ and $h, w > 0$, denote
$$
Q_\bbR(x_0, h, w) := \{z = x + yi \; | \; |x-x_0| < w \text{ and } |y| < h\}.
$$
Additionally, let
$$
Q_{\pbbD}(\ixp(x_0), h, w) := \ixp(Q_\bbR(x_0, h, w)).
$$

Given $1 < \lambda < \Lambda(n)$, let $\eta$ and $\mu$ be the constants in \propref{a priori geometry}. Consider the set
$$
Q_\pm := Q_\pbbD(e_\pm, \eta|J|, 2\eta|J|/\mu) \cap (\bbC \setminus \overline{U}).
$$
Then $Q_\pm$ is a quadrilateral whose boundary consists of three smooth arcs $\partial_{\tp} Q_\pm$, $\partial_{\bt} Q_\pm$ and $\partial_{\ot} Q_\pm$, and one piecewise smooth arc $\partial_{\In}Q_\pm$ such that
\begin{itemize}
\item $\partial_{\tp}Q_\pm$ and $\partial_{\bt}Q_\pm$ are contained in circles centered at $0$ of radii $e^{2\pi h}$ and $e^{-2\pi h}$ respectively;
\item $\partial_{\ot}Q_\pm$ is contained in a radial line; and
\item $\partial_{\In}Q_\pm \ni e_\pm$ is contained in $\cE_\pm$.
\end{itemize}

Let $W$ be the component of the complement of $\cQ_2 \cup \cQ_{1/2} \cup \cE_+ \cup \cE_-$ that contains $\pbbD \setminus \overline{J}$. Let $\hGamma^\pbbD_\pm(U, \lambda)$ be a path family in $W$ such that any path $\gamma\in \hGamma^\pbbD_\pm(U, \lambda)$ has one endpoint in $\cE_\pm \setminus \partial_{\In}Q_\pm$ and the other endpoint in $J[\lambda]_\pm$. Additionally, let $\chGamma^\pbbD_\pm(U, \lambda)$ be a path family in $W$ such that any path $\gamma\in \chGamma^\pbbD_\pm(U, \lambda)$ has one endpoint in $\partial_{\In}Q_\pm$ and the other endpoint in $J[\lambda]_\pm$, and $\gamma$ contains a path in $\Gamma_{Q_\pm}(\partial_{\In}Q_\pm, \partial Q_\pm \setminus \partial_{\In}Q_\pm)$. See \figref{fig:path to circle}. Define
\begin{equation}\label{eq:path to circle}
\Gamma^\pbbD(U, \lambda) := \hGamma^\pbbD_-(U, \lambda)\cup \chGamma^\pbbD_-(U, \lambda) \cup \hGamma^\pbbD_+(U, \lambda) \cup \chGamma^\pbbD_+(U, \lambda).
\end{equation}

\begin{figure}[h]
\centering
\includegraphics[scale=0.35]{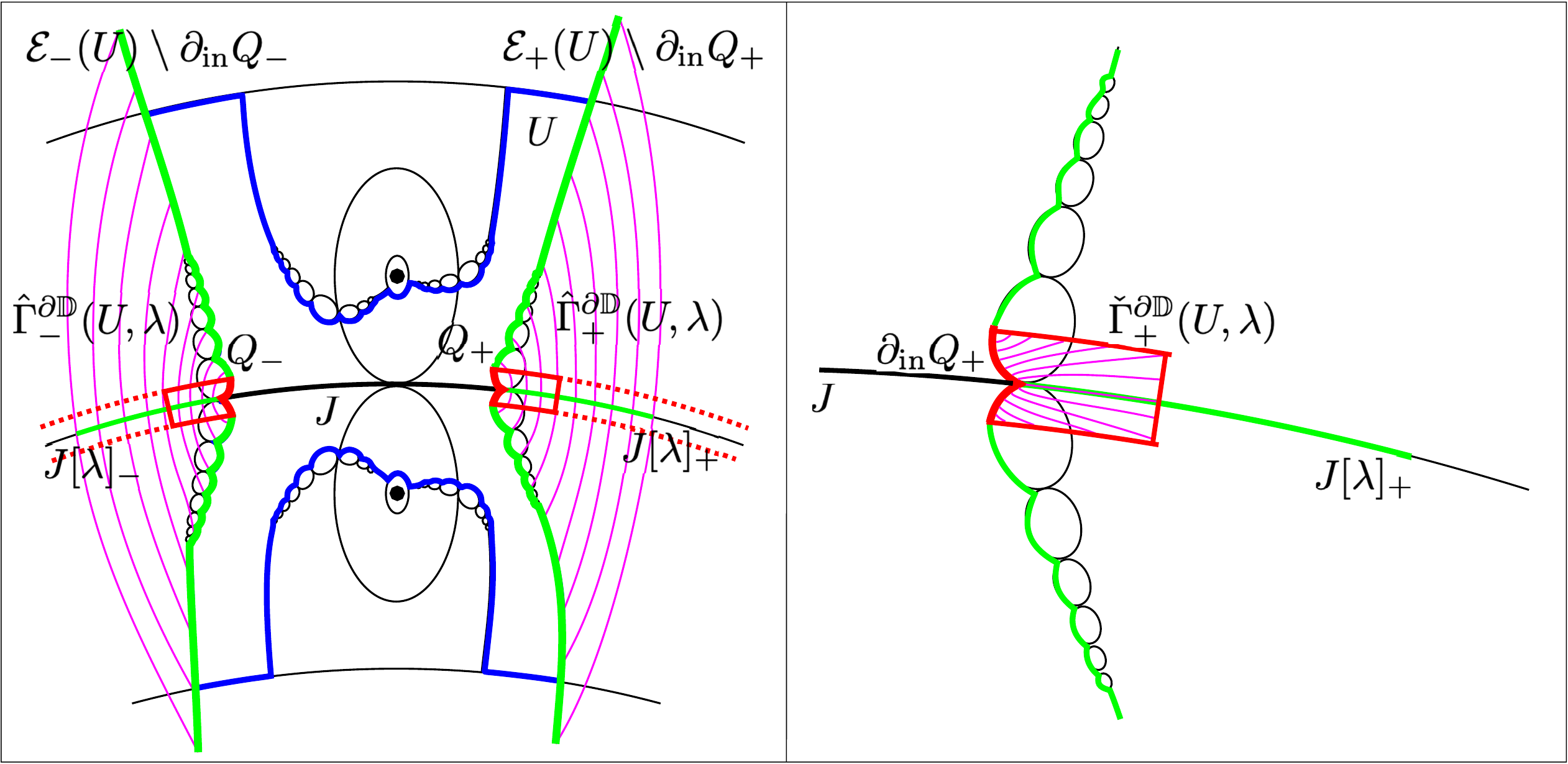}
\caption{Left: the path families $\hGamma^\pbbD_\pm(U, \lambda)$ consisting of paths from $\cE_\pm(U) \setminus \partial_{\In} Q_\pm$ to $J[\lambda]_\pm$. Right: the path family $\chGamma^\pbbD_+(U, \lambda)$ consisting of paths from $\partial_{\In} Q_+$ to $J[\lambda]_+$ overflowing paths from $\partial_{\In} Q_+$ to $\partial Q_+ \setminus \partial_{\In} Q_+$.}
\label{fig:path to circle}
\end{figure}

\begin{prop}[Lower bound on extremal lengths from edges to $\pbbD$]\label{length to circle}
Suppose $k \leq q_{n+m}$ for some $m \geq 1$. For $1< \lambda < \Lambda(n)$, there exists a uniform constant $C = C(m, \lambda) > 0$ independent of $n$ such that
$$
\cL(\Gamma^\pbbD(U, \lambda)) > C.
$$
\end{prop}

\begin{proof}
The fact that $\cL(\hGamma^\pbbD_\pm(U, \lambda))$ has a uniform lower bound independent of $n$ follows immediately from \propref{a priori geometry} and \lemref{dist to mod}.

Let
$$
\Gamma_\pm := \Gamma_{Q_\pm}(\partial_{\In} Q_\pm, \partial_{\ot}Q_\pm) \cup \Gamma_W(\partial_{\tp}Q_\pm, J[\lambda]_\pm) \cup \Gamma_W(\partial_{\tp}Q_\pm, J[\lambda]_\pm).
$$
It is easy to see that each path family on the right-hand side has a uniform lower bound on its extremal length independent of $n$ by \lemref{dist to mod}, and hence, so does $\Gamma_\pm$ by \lemref{union width}. Observe that $\chGamma^\pbbD_\pm(U, \lambda)$ overflows $\Gamma_\pm$. Hence, by \lemref{overflow}, $\cL(\chGamma^\pbbD_\pm(U, \lambda))$ also has a uniform lower bound independent of $n$.

The lower bound on $\cL(\Gamma^\pbbD(U, \lambda))$ now follows from another application of \lemref{union width}.
\end{proof}

\section{Local Connectivity at a Critical Point}\label{sec:lc at crit}

Consider the puzzle disks $D^n$ for $n \geq n_0$ constructed in \secref{sec:disks}. For concreteness, we assume that $n_0$ is even, so that $n_0 = 2\bn_0$ for some $\bn_0 \geq 0$. For $n \geq n_0 +2$, define the {\it puzzle annulus of level $n$} as
$$
A^n := D^{n-2} \setminus \overline{D^n}.
$$
By \propref{nested disks}, $A^n$ is non-degenerate (i.e. $\mod A^n >0$). Moreover, by \propref{fiber in disks}, the sequence of nested annuli $\{A^n\}_{n=n_0}^\infty$ surrounds the fiber $X_0 \subset J_F$ rooted at the critical point $c_0 = 1$.

In this section we prove the following theorem.

\begin{thm}[Triviality of $X_0$]\label{lc at crit}
There exists a uniform constant $\epsilon >0$ such that
$$
\mod A^n > \epsilon
\matsp{for}
n \geq n_0.
$$
Consequently, $X_0 = \{c_0\}$.
\end{thm}

\subsection{Doubled puzzle annuli}

Before proving \thmref{lc at crit}, we need a preliminary result relating the moduli of successive puzzle annuli.

For $n \geq n_0 + 4$, define the {\it doubled puzzle annulus of level $n$} as
$$
\bfA^n := D^{n-4} \setminus \overline{D^n}.
$$
We show that $A^n$ contains a pullback of $\bfA^{n-2}$ along $\pbbD$ under a map with uniformly bounded degree.

Recall that we have
$$
F^{r_n}(D^n) = D^{n-1}|_{J^+_n} \subset D^{n-2}.
$$

\begin{lem}\label{collar base}
There exist uniform constants $\delta >0$ and $\lambda >0$ independent of $n$ such that the following holds. Let 
$$
\hJ_n := g^{r_n}(J^-_{n-2}).
$$
Then $J^-_{n-1}[\delta] \subset \hJ_n \subset J^-_{n-3}$ and $J^-_{n-6} \subset \hJ_n[\lambda]$.
\end{lem}

\begin{proof}
We have
$$
(0, q_{n-2}-q_n+2q_{n+1})_c \supset (0, -q_{n-1}+2q_{n+1})_c \supset (0, q_n+q_{n+1})_c = (0, r_n)_c. 
$$
Hence,
$$
(-q_{n-2}, -q_n+2q_{n+1})_c \supset (-q_{n-2}, -q_{n-2} +r_n)_c.
$$
Thus,
$$
\hJ_n = (-q_{n-2}+r_n, -q_{n-1}+r_n)_c \supset (-q_n+2q_{n+1}, -q_{n-1}+r_n)_c \Supset J^-_{n-1}.
$$
By \corref{hermancor}, there exists a uniform constant $\delta>0$ such that the first containment holds.

Observe that
$$
c_{r_n} \in (0, q_n)_c.
$$
Hence,
$$
c_{-q_{n-2}+r_n} \in (-q_{n-2}, 0)_c.
$$
Moreover,
$$
c_{-q_{n-1}+r_n} \in (0, -q_{n-1}+q_n)_c \subset (0, q_{n-2})_c \subset (0, -q_{n-3})_c.
$$
Thus, the second containment holds.

Since $\hJ_n \supset J^-_{n-1}$,  \corref{hermancor} implies the last containment holds for some uniform constant $\lambda >0$.
\end{proof}

\begin{prop}\label{collar pullback}
Let $\hU$ be the $r_n$th pullback of $D^{n-6}|_{\hJ_n}$ along $\pbbD$. Then $D^n \Subset \hU \subset D^{n-2}$. Moreover, $F^{r_n}|_{\hU}$ has a uniformly bounded degree $d'$.
\end{prop}

\begin{proof}
By \lemref{collar base}, we have
$$
F^{r_n}(D^n) = D^{n-1}|_{J^+_n} \Subset D^{n-6}|_{\hJ_n},
$$
Hence, $D^n \Subset \hU$. The second containment, $\hU \subset D^{n-2}$, follows from \propref{disk pull in small}. The last claim follows from \lemref{exp time growth} i) and \propref{crit in disk}.
\end{proof}

\begin{prop}\label{thick in thin}
There exist uniform constants $\epsilon_0, C>0$ such that for $n \geq n_0 +4$, we have
$$
\mod A^n > \min\{\epsilon_0, \; C\mod \bfA^{n-2}\}.
$$
\end{prop}

\begin{proof}
By \lemref{collar base}, the slitted annulus
$$
A := (D^{n-6}\setminus \overline{D^{n-1}})|_{\hJ_n}
$$
is non-degenerate. Its modulus is equal to the extremal length of the following path family:
$$
\Gamma := \Gamma_A(\partial D^{n-1}, \partial D^{n-6} \cup (J^-_{n-6} \setminus \hJ_n)).
$$

Let $\Gamma^\pbbD \subset \Gamma$ be the path family such that $\gamma \in \Gamma^\pbbD$ has one endpoint in $\partial D^{n-1}$ and the other endpoint in $J^-_{n-6} \setminus \hJ_n$. Then
$$
\Gamma = \Gamma_A(\partial D^{n-1}, \partial D^{n-6}) \cup \Gamma^\pbbD.
$$

Let $\delta, \lambda>0$ be the uniform constants given in \lemref{collar base}. Since $J^-_{n-1}[\delta] \subset \hJ_n$, we see that $\Gamma^\pbbD$ overflows the path family $\Gamma^\pbbD(D^{n-1}, \lambda)$ defined in \eqref{eq:path to circle}. By \lemref{overflow} and \propref{length to circle}, there exists a uniform constant $\omega >0$ such that
$$
\cW(\Gamma^\pbbD) \leq \omega.
$$

Clearly,
$$
\cL(\Gamma_A(\partial D^{n-1}, \partial D^{n-6})) \geq \mod \bfA^{n-2}.
$$
Thus, by \lemref{union width}, we have
$$
\cW(\Gamma) \leq \frac{1}{\mod \bfA^{n-2}} + \omega.
$$
This implies that
$$
\mod A > \min\{\epsilon', C' \mod \bfA^{n-2}\}
$$
for some uniform constants $\epsilon', C'>0$. The result now follows from \lemref{pullback mod}.
\end{proof}

\begin{cor}\label{thick if thin}
The sequence $\{\mod \bfA^n\}_{n=n_0 + 4}^\infty$ has a uniform positive lower bound if and only if $\{\mod A^n\}_{n=n_0 + 2}^\infty$ does.
\end{cor}

\subsection{Applying the covering lemma}

To prove \thmref{lc at crit}, we need the following crucial analytic estimate obtained by Kahn and Lyubich in \cite{KaLy2} (compare with \lemref{pullback mod}).

\begin{thm}[Covering lemma]\label{cover lem}
Let $U'' \Subset U' \Subset U$ and $V'' \Subset V' \Subset V$ be topological disks, and let $G : (U, U', U'') \to (V, V', V'')$ be a holomorphic branched covering between respective disks. Denote $d_{\bg} = \deg G \geq d_{\sm} = \deg(G|_{U'})$. Suppose for some $\kappa > 0$, we have the following \emph{collar condition}:
\begin{equation}\label{eq:collar}
\mod(V' \setminus \overline{V''}) > \kappa \mod (U \setminus \overline{U''}).
\end{equation}
Then there exists a uniform constant $\epsilon_1 = \epsilon_1(\kappa, d_{\bg})>0$ such that either
\begin{equation}\label{eq:cover conc 1}
\mod(U \setminus \overline{U''}) > \epsilon_1
\end{equation}
or
\begin{equation}\label{eq:cover conc 2}
\mod(U \setminus \overline{U''}) > \frac{\kappa}{2d_{\sm}^2}\mod(V \setminus \overline{V''}).
\end{equation}
\end{thm}

To apply \thmref{cover lem}, we use the following setup. Choose a large even number $N = 2\bN >> 1$ to be specified later. Given an even number $n = 2\bn \geq n_0+N$, let
\begin{equation}\label{eq:setup 1}
U := D^{n-4}
\matsp{,}
U' := D^{n-2}
\matsp{,}
U'' := D^n
\matsp{and}
G := F^{\bfr_{n-4} - \bfr_{n-N}}|_{D^{n-4}}.
\end{equation}
Then
\begin{equation}\label{eq:setup 2}
V = G(D^{n-4}) = D^{n-N}|_{g^{\bfr_{n-4} - \bfr_{n-N}}(J^-_{n-4})}
\hspace{2mm} , \hspace{2mm}
V' = G(D^{n-2})
\hspace{2mm} \text{and} \hspace{2mm}
V'' = G(D^n).
\end{equation}
Observe that we have
$$
\bfA^n = U \setminus \overline{U''}
\matsp{and}
A^n = U' \setminus \overline{U''}.
$$

\begin{lem}\label{cover degrees}
Let $d_{\bg} = \deg(G)$ and $d_{\sm} = \deg(G|_{D^{n-2}})$. Then there exist $C_{\bg} = C_{\bg}(N) \geq 3$ independent of $n$, and $C_{\sm} \geq 3$ independent of $n$ and $N$ such that $d_{\bg} < C_{\bg}$ and $d_{\sm}<C_{\sm}$.
\end{lem}

\begin{proof}
The bound on $d_{\bg}$ is an immediate consequence of \propref{crit in disk}. The bound on $d_{\sm}$ follows from \lemref{exp time growth} iv) and \propref{crit in disk}.
\end{proof}

Let us outline the proof of \thmref{lc at crit}. Suppose towards a contradiction that $\mod A^n$ has no uniform lower bound. Then by \corref{thick if thin}, neither does $\mod \bfA^n$. Hence, we may assume, without loss of generality, that the following degeneracy condition holds for some arbitrarily small number $\epsilon > 0$:
\begin{equation}\label{eq:min assu}
\mod \bfA^n = \min_{n_0 \leq k \leq n}\mod \bfA^k < \epsilon.
\end{equation}
If $\epsilon$ is sufficiently small, then \eqref{eq:min assu} together with \propref{thick in thin} and \lemref{pullback mod} imply that the collar condition \eqref{eq:collar} holds for some uniform constant $\kappa >0$ independent of $n$ and $N$.

Applying \thmref{cover lem}, we conclude that either \eqref{eq:cover conc 1} or \eqref{eq:cover conc 2} must hold. However, \eqref{eq:cover conc 1} directly contradicts \eqref{eq:min assu}. Moreover, we can show that for any even number $k$ such that $n - K + 4 \leq k \leq n - 4$, the annuli $V \setminus \overline{V''}$ contains a pullback of $A^k$ along $\pbbD$ under a map with uniformly bounded degree (\propref{image contains pullbacks}). This implies that if $N$ is sufficiently large, then \eqref{eq:cover conc 2} also contradicts \eqref{eq:min assu}. Therefore, \eqref{eq:min assu} cannot be true, and $\mod A^n$ must have a uniform lower bound.

The principle difficulty is that slits have to be cut into puzzle disks before they can be pulled back along $\pbbD$. This procedure decreases the moduli of the puzzle annuli involved, potentially ruining the argument outlined above. However, using \propref{length to circle}, we can show that if a puzzle annulus is already nearly degenerate (as assumed in \eqref{eq:min assu}), then cutting slits into it does not significantly impact its moduli.

\subsection{Pulling back puzzle annuli to $V \setminus \overline{V''}$}

Let $k = 2\bk$ be an even number such that $n - K + 4 \leq k \leq n - 4$.

\begin{lem}\label{R pullback times}
Define
$$
R^0_k := \bfr_k - \bfr_{n-N},
$$
$$
R^1_k := q_{k+2} - \bfr_{k-2} + \bfr_{n-N}
$$
and
$$
R^2_k := q_{k+4} - r_k - r_{k-1} - q_{k+2}.
$$
Then we have:
\begin{enumerate}[i)]
\item $q_{k+4} = R^0_k + R^1_k + R^2_k$;
\item either $a_{k +i} = 1$ for $0 \leq i \leq 3$ and $R^2_k = 0$, or $R^2_k \geq q_k > r_{k-4}$; and
\item
$$
R^1_{k+2} = R^1_k + R^2_k = R^1_{n-N+4} + \sum_{i=\bn-\bN+2}^{\bk} R^2_{2i}.
$$
\end{enumerate}
\end{lem}

\begin{proof}
Claim i) is obvious.

By \lemref{exp time growth} i) and ii), we have
$$
q_{k+4} \geq r_{k+1} + r_k = q_{k+1} + q_{k+2} + r_k \geq r_{k-1} + q_{k+2} + r_k,
$$
where the equality holds if and only if $a_{k +i} = 1$ for $0 \leq i \leq 3$. Furthermore, it is easy to check that if $a_{k+i} \geq 2$ for some $1\leq i \leq 3$, then $R^2_k \geq q_{k+i}$. Claim ii) follows.

In claim iii), the first equality is obvious, and the second equality can be checked by a straightforward induction.
\end{proof}

Consider the orbit of $J^-_k$ under $g^{q_{k+4}}$. We decompose
$$
g^{q_{k+4}} = g^{R^2_k} \circ g^{R^1_k} \circ g^{R^0_k},
$$
and denote
$$
J^i_k := g^{R^i_k}(J^{i-1}_k)
$$
for $i \in \{0, 1, 2\}$ (letting $J^{-1}_k = J^-_k$). Also define
$$
\hJ^2_k := g^{R^2_k}(J^-_{k-4}) \cap J^-_{k-4}.
$$

\begin{lem}\label{stage bases}
We have
\begin{enumerate}[i)]
\item $J^1_k \Subset g^{q_{k+2}}(J^+_{k-1}) \Subset J^-_{k-2}$,
\item $J^2_k \Subset J^-_{k-1}$, and
\item $J^-_{k-2} \Subset \hJ^2_k$. Consequently, $\hJ^2_k \neq \varnothing$.
\end{enumerate}
\end{lem}

\begin{proof}
For i), observe that $g^{r_k}(J^-_k) = J^+_k \Subset J^-_{k-1}$ and
$$
g^{r_{k-1}}(J^+_k) \Subset g^{r_{k-1}}(J^-_{k-1}) = J^+_{k-1} = (q_{k-1}, q_k)_c \Subset J^-_{k-2} = (-q_{k-2}, -q_{k-1})_c.
$$
Since
$$
c_{q_{k+2}} \in(0, -q_k-q_{k-1})_c
$$
we have
$$
J^1_k \Subset g^{q_{k+2}}(J^+_{k-1}) \Subset J^-_{k-2}.
$$

For ii), we have
$$
g^{q_{k+4}}(J^-_k) = (-q_k+q_{k+4}, -q_{k+1} + q_{k+4})_c \Subset (-q_k, -q_{k-1})_c.
$$

For iii), we have $g^{R^2_k}$ mapping $c_{r_k+r_{k-1}+q_{k+2}}$ to $c_{q_{k+4}}$. Since
$$
c_{r_k+r_{k-1}+q_{k+2}} \in g^{r_k+r_{k-1}+q_{k+2}}(-q_k, -q_{k+1})_c \Subset (q_{k-1}+q_{k+2}, q_k+q_{k+2})_c,
$$
we see that
$$
c_{r_k+r_{k-1}+q_{k+2}} \in (q_{k-1} + q_k, q_k)_c.
$$
We have either
$$
c_{r_k + r_{k-1}+q_{k+2}} \in (q_{k-1}+q_k, 0)_c
\matsp{or}
c_{r_k + r_{k-1}+q_{k+2}} \in (q_{k+4}, q_k)_c.
$$
Hence, either
$$
(r_k + r_{k-1}+q_{k+2} - q_{k+4}, 0)_c \subset (q_{k-1}, 0)_c
$$
or
$$
(0, r_k + r_{k-1}+q_{k+2} - q_{k+4})_c \subset (0, q_k)_c.
$$
In either case, the claim follows from the fact that
$$
J^-_{k-2} = (-q_{k-2}, -q_{k-1})_c \Subset (-q_{k-4} - q_{k-1}, -q_{k-3}-q_k)_c \subset \hJ^2_k.
$$
\end{proof}

Let $U^2_k := D^{k-4}|_{\hJ^2_k}$. Define $U^1_k$ as the $R^2_k$th pullback of $U^2_k$ along $\pbbD$. By \propref{edge},
$$
\hJ^1_k := g^{-R^2_k}(\hJ^2_k) = J^-_{k-4} \cap g^{-R^2_k}(J^-_{k-4})
$$
is both the base and the full base of $U^1_k$, so that
$$
U^1_k \cap \pbbD = \hJ^1_k.
\matsp{and}
\overline{U^1_k} \cap \pbbD = \overline{\hJ^1_k}.
$$

\begin{prop}\label{stage 1 pullback}
We have
$$
U^1_k \subset D^{k-4}
\matsp{and}
U^1_k \Subset U^2_{k-2}.
$$
\end{prop}

\begin{proof}
The first containment is an immediate consequence of \lemref{R pullback times} ii) and \propref{disk pull in}. By \propref{nested disks} and \lemref{stage bases} iii), we have
$$
U^1_k \subset D^{k-4} \Subset D^{k-6}
\matsp{and}
\hJ^1_k \subset J^-_{k-4} \Subset \hJ^2_{k-2} \subset J^-_{k-6}.
$$
The second containment follows.
\end{proof}
 
Define $U^0_k$ as the $R^1_k$th pullback of $U^1_k$ along $\pbbD$. Then 
$$
\hJ^0_k := g^{-R^1_k}(\hJ^1_k)
$$
is both the base and the full base of $U^0_k$, so that
$$
U^0_k \cap \pbbD = \hJ^0_k.
\matsp{and}
\overline{U^0_k} \cap \pbbD = \overline{\hJ^0_k}.
$$
Observe that by \lemref{stage bases} ii) and iii), we have
$$
J^i_k \Subset \hJ^i_k
\matsp{for}
i \in \{0, 1, 2\}.
$$

\begin{prop}\label{image contains pullbacks}
Let $k = 2\bk$ be an even number such that $n-N + 6 \leq k \leq n-4$. Then
$$
U^0_k \Subset U^0_{k-2} \subset D^{n-N}.
$$
\end{prop}

\begin{proof}
Recall that $U^0_k$ and $U^0_{k-2}$ are the $R^1_k$th and $(R^1_{k-2}+R^2_{k-2})$th pullback of $U^1_k$ and $U^2_{k-2}$ along $\pbbD$ respectively. Hence, the first containment follows from \lemref{R pullback times} iii) and \propref{stage 1 pullback}.

We have
$$
U^2_{n-N+4} = D^{n-N}|_{\hJ^2_{n-N+4}}
$$
Observe that
$$
R^1_{n-N+2} = q_{n-N+4} \geq r_{n-N+2} > r_{n-N},
$$
where in the last inequality, we used \lemref{exp time growth} i). Since $U^0_{n-N+4}$ is the $R^1_{n-N+2}$th pullback of $U^2_{n-N+4}$, the second containment now follows from \propref{disk pull in}.
\end{proof}

\subsection{Modulus of $V \setminus \overline{V''}$}

Let $k = 2\bk$ be an even number such that $n - K + 6 \leq k \leq n - 4$, and let $\Lambda(n) >1$ be the constant given in \lemref{big lambda}.

\begin{lem}\label{lambda cover}
There exist uniform constants $\delta >0$ and $1 < \hlambda < \Lambda(n)$ such that
\begin{enumerate}[i)]
\item $J^-_{k-2}[\delta] \Subset \hJ^2_k$;
\item $J^-_{k-6} \Subset \hJ^1_k[\hlambda/2]$; and
\item $\hJ^0_{k-2} \Subset \hJ^0_k[\hlambda/2]$.
\end{enumerate}
\end{lem}

\begin{proof}
The result follows immediately from \lemref{stage bases} iii) and \corref{hermancor}.
\end{proof}

Let $V$, $V'$ and $V''$ be the topological disks given in \eqref{eq:setup 2}, and consider the path family
$$
\Gamma := \Gamma_{V \setminus \overline{V''}}(\partial V'', \partial V).
$$
Then we have $\mod(V \setminus \overline{V''}) = \cL(\Gamma)$.

By \propref{image contains pullbacks}, the set $U_{k-2}^0 \setminus \overline{U_k^0}$ is a non-degenerate annulus. Its modulus is equal to the extremal length of the following path family:
$$
\Gamma^0_k := \Gamma_{U_{k-2}^0 \setminus \overline{U_k^0}}(\partial U_k^0, \partial U_{k-2}^0).
$$
Recall that $\Gamma^\pbbD(U_k^0, \hlambda)$ is a family of paths connecting bounding edges of $U_k^0$ to the arcs $\hJ^0_k(\hlambda)_\pm$ (see \eqref{eq:path to circle}). Define
$$
\tiGamma^0_k := \Gamma^0_k \cup \Gamma^\pbbD(U_k^0, \hlambda).
$$

\begin{lem}\label{base inner component}
Let $J''$ be the full base of $V''$. Then there exist uniform constants $\lambda >0$ and $\delta >0$ such that
\begin{enumerate}[i)]
\item $\hJ^0_{n-4}\Subset J''[\lambda]$; and
\item $J''[\delta] \Subset g^{\bfr_{n-4}-\bfr_{n-N}}(J^-_{n-4})$.
\end{enumerate}
\end{lem}

\begin{proof}
Denote
$$
r := r_n - \bfr_{n-4}+\bfr_{n-N}.
$$
Then by \lemref{exp time growth} i) and iv), we have $0 < r < q_{n+2}$. Observe that $V''$ is the $r$th pullback of $D^{n-1}|_{J^+_n}$ along $\pbbD$. Thus, we have
$$
g^{\bfr_{n-4}+\bfr_{n-N}}(J^-_n) \subset J'' \subset g^{-r}(J^-_{n-1}).
$$
Since
$$
\hJ^0_{n-4} \subset g^{-R^1_{n-4} - R^2_{n-4}}(J^-_{n-8}),
$$
claim i) follows from \corref{hermancor}.

Note
$$
J^+_n = (q_{n+1}, q_n)_c \Subset J^-_{n-1} = (-q_n, -q_{n-1})_c \subset (-q_n+q_{n+1}, -q_{n-1})_c.
$$
Taking the preimage under $g^{-r_n}$, we obtain
$$
g^{-r_n}(J^-_{n-1}) \Subset (-2q_n, -q_{n-1} -r_n)_c \Subset (-q_{n-2}, -q_{n-1})_c = J^-_{n-2}.
$$
Hence
$$
J'' \Subset g^{\bfr_{n-4}-\bfr_{n-N}}(J^-_{n-2}) \Subset g^{\bfr_{n-4}-\bfr_{n-N}}(J^-_{n-4}).
$$
Claim ii) now follows from \corref{hermancor}.
\end{proof}

\begin{prop}\label{cover length decomp}
There exist uniform constants $\epsilon_0, C>0$ independent of $n = 2\bn$ and $N = 2\bN$ such that
$$
\cL(\Gamma) > \min\left\{\epsilon_0, \; C \sum_{\bk=\bn-\bN+3}^{\bn-2} \cL(\tiGamma^0_{2\bk})\right\}.
$$
\end{prop}

\begin{proof}
Recall that
$$
V = G(D^{n-4}) = D^{n-N}|_{g^{\bfr_{n-4} - \bfr_{n-N}}(J^-_{n-4})}
$$
Let $\lambda >0$ be the constant given in \lemref{base inner component}. By \propref{length to circle}, there exists a uniform constant $\omega >0$ independent of $n$ such that
$$
\cW(\Gamma^\pbbD(V'', \lambda)) < \omega.
$$

Denote
$$
L := \sum_{\bk=\bn-\bN+3}^{\bn-2}\cL(\tiGamma^0_{2\bk}).
$$
By \lemref{lambda cover} iii) and \lemref{base inner component} ii), we see that $\Gamma$ disjointly overflows $\{\Gamma^\pbbD(V'', \lambda)\} \cup\{\tiGamma^0_{2\bk}\}_{\bk=\bn-\bN+3}^{\bn-2}$ Thus, by \lemref{overflow} and \lemref{union width}, we have
$$
\cW(\Gamma) \leq L^{-1} +\omega.
$$
The result follows.
\end{proof}

\begin{prop}\label{slit 0}
There exist uniform constants $\epsilon_0, C > 0$ such that
$$
\cL(\tiGamma^0_k) > \min \{\epsilon_0, \; C\cL(\Gamma^0_k)\}.
$$
\end{prop}

\begin{proof}
Recall that $U_k^0$ is the $(R^1_k + R^2_k)$th pullback of $D^{k-4}|_{\hJ^2_k}$ along $\pbbD$. By \lemref{R pullback times} i) and \propref{length to circle}, there exists a uniform constant $\omega >0$ independent of $n$ such that
$$
\cW(\Gamma^\pbbD(U_k^0, \hlambda)) < \omega.
$$
Applying \lemref{union width}, we have
$$
\cW(\tiGamma^0_k) < \frac{1}{\cL(\Gamma^0_k)} + \omega.
$$
The result follows
\end{proof}

\begin{prop}\label{slit 1}
There exist uniform constants $\epsilon_0, C > 0$ such that
$$
\cL(\Gamma^0_k)  > \min \{\epsilon_0, \; C\mod A^{k-4}\}.
$$
\end{prop}

\begin{proof}
Consider the pair of nested disks $U^0_k \Subset U^0_{k-2}$ and $U^1_k \Subset U^0_{k-2} = D^{k-6}|_{\hJ^2_{k-2}}$ (see Proposition \ref{stage 1 pullback} and \ref{image contains pullbacks}). The map $H : (U^0_k, U^0_{k-2}) \to (U^1_k, U^0_{k-2})$ defined by
$$
H:= F^{R^1_{k-2} + R^2_{k-2}} = F^{R^1_k}
$$
(see \lemref{R pullback times} iii)) is a branched covering between respective disks. By \propref{crit in disk}, \lemref{R pullback times} i) and \lemref{pullback mod}, there exists a uniform constant $C'>0$ independent of $n$ such that
$$
\cL(\Gamma^0_k) = \mod(U_{k-2}^0 \setminus \overline{U_k^0}) > C'\mod(U^0_{k-2} \setminus \overline{U^1_k}).
$$
The modulus of $U^0_{k-2} \setminus \overline{U^1_k}$ is equal to the extremal length of the following path family
$$
\Gamma_k := \Gamma_{U^0_{k-2} \setminus \overline{U^1_k}}(\partial U^1_k, \partial U^1_k).
$$

Denote
$$
\Gamma^1_k := \Gamma_{D^{k-6} \setminus \overline{U_k^1}}(\partial U_k^1, \partial D^{k-6}).
$$
Then
$$
\cL(\Gamma^1_k) > \mod A^{k-4}.
$$
Define
$$
\tiGamma^1_k := \Gamma^1_k \cup \Gamma^\pbbD(U_k^1, \hlambda).
$$

Recall that $U_k^1$ is the $R^1_k$th pullback of $D^{k-4}|_{\hJ^2_k}$. By \lemref{R pullback times} i) and \propref{length to circle}, there exists a uniform constant $\omega >0$ independent of $n$ such that
$$
\cW(\Gamma^\pbbD(U_k^1, \hlambda)) < \omega.
$$
Applying \lemref{union width}, we have
$$
\cW(\tiGamma^1_k) < \frac{1}{\cL(\Gamma^1_k)} + \omega.
$$

Finally, observe that by \lemref{lambda cover} i) and ii), the path family $\Gamma_k$ overflows $\tiGamma^1_k$. The result follows from \lemref{overflow}.
\end{proof}

\subsection{Proof of the triviality of $X_0$}

We are now ready to prove the main result of this section.

\begin{proof}[Proof of \thmref{lc at crit}]
Choose a large even number $N = 2\bN >>1$ to be specified later. Let $n \geq n_0 + N$. For concreteness, assume that $n = 2\bn$ is even. Throughout this proof, let $C >0$ stand for a uniform constant independent of $n$ and $N$.

Denote
$$
M := \mod(U \setminus \overline{U''}) = \mod \bfA^n > 0.
$$
Assume that \eqref{eq:min assu} holds for some sufficiently small $\epsilon$. Then by \propref{thick in thin}, we have
$$
\mod A^{n-2} > \min\{\epsilon_0, C\mod \bfA^{n-4}\} > CM.
$$
Hence, \eqref{eq:collar} holds with $\kappa = C$. Since \eqref{eq:cover conc 1} contradicts \eqref{eq:min assu}, \thmref{cover lem} implies that \eqref{eq:cover conc 2} holds.

By Proposition \ref{slit 0}, \ref{slit 1} and \ref{thick in thin}, we see that
$$
\cL(\tiGamma^0_k) > \min\{\epsilon_0, C\mod A^{k-4}\} > CM.
$$
for every even number $k = 2\bk$ such that $\bn - \bN +3 \leq \bk \leq \bn - 2$. Then by \propref{cover length decomp}, we have
$$
\mod(V\setminus \overline{V''}) = \cL(\Gamma) > \min\left\{\epsilon_0, \; C \sum_{\bk=\bn-\bN+3}^{\bn-2} \cL(\tiGamma^0_{2\bk})\right\} > CM(\bN-5).
$$
Using \lemref{cover degrees}, we see that $\bN$ can be made arbitrarily large without increasing $d_{\sm}$. This contradicts \eqref{eq:cover conc 2}.

Thus, there is some uniform lower bound on $\mod\bfA^n$. By \corref{thick if thin}, the same is true for $\mod(A^n)$. Since $A^n$ surrounds $X_0 \ni c_0$, we conclude by \lemref{grot} that $X_0 = \{c_0\}$.
\end{proof}

\section{Spreading Local Connectivity}\label{sec:lc on circ}

In \secref{sec:lc at crit}, we proved that the fiber $X_0$ rooted at the critical point $\xi_0 = c_0$ is trivial. To complete the proof of Theorem A stated in \secref{sec:intro}, we need to extend this result to fibers $X_s$ rooted at arbitrary points $\xi_s \in \pbbD$ with angles $s \in \bbR/\bbZ$.

\subsection{Combinatorial address of $s$}

For $n \in \bbN$, denote
$$
g_n := g^{-q_n}|_{I^-_{n-1}}.
$$
Let $\sigma_n = (\alpha_n, \beta_n)$ for some $0\leq \alpha_n < a_n$ and $\beta_n \in \{0,1\}$. Denote
$$
g^{\sigma_n} := g_n^{\alpha_n}\circ g_{n-1}^{\beta_n}.
$$
The inverse of $g^{\sigma_n}$ is denoted by $g^{-\sigma_n}$.

Let $\xi_s \in J^-_n = (-q_n, -q_{n+1})_c \subset \pbbD$, and assume that $\xi_s$ is not an iterated preimage of $\xi_0$. 

\begin{lem}\label{codes}
There exists a unique pair $\sigma_{n+1}(s) = (\alpha_{n+1}(s), \beta_{n+1}(s))$ such that
$$
g^{-\sigma_{n+1}(s)}(\xi_s) \in J^-_{n+1}.
$$
\end{lem}

\begin{proof}
The intervals $I^-_{n+2}, I^-_{n+1}, g_n(I^-_{n+1}), g_{n+1} \circ g_n(I^-_{n+1}), \ldots, g_{n+1}^{a_{n+1}-1} \circ g_n(I^-_{n+1})$ have pairwise disjoint interiors, and they cover $J^-_n$ except iterated preimages of $\xi_0$. Thus, $\xi_s$ belongs to exactly one of these arcs, and there is a unique pair $\sigma_{n+1}(s)$ such that $g^{-\sigma_{n+1}(s)}$ brings this arc back to $J^-_{n+1}$.
\end{proof}

For $k \geq 0$, inductively define $s_k$ and $\sigma_{n+k+1}(s_k)$ by
$$
s_0 := s
\matsp{and}
\xi_{s_{k+1}} := g^{-\sigma_{n+k+1}(s_k)}(\xi_s) \in J^-_{n+k+1}.
$$
For $m \geq 1$, the {\it $(n, m)$th combinatorial address of $s$} is defined as the following $m$-tuple of pairs
$$
\Sigma^n_{n+m}(s) = (\sigma_{n+1}(s_0), \ldots, \sigma_{n+m}(s_{m-1})).
$$
We denote
$$
g^{\Sigma^n_{n+m}(s)} := g^{\sigma_{n+1}(s_0)} \circ \ldots \circ g^{\sigma_{n+m}(s_{m-1})}.
$$
The inverse of $g^{\Sigma^n_{n+m}}$ is denoted by $g^{-\Sigma^n_{n+m}}$. Lastly, we define $\Sigma^n_n(s)$ to be the trivial $0$-tuple. The following result is obvious.

\begin{lem}\label{comb add}
For $n \leq k \leq m$, let
$$
\xi_{s'} := g^{-\Sigma^n_k(s)}(\xi_s) \in J^-_k,
$$
and
$$
\Sigma^n_k(s) = (\sigma_{n+1}, \sigma_{n+2}, \ldots, \sigma_k)
\matsp{and}
\Sigma^k_m(s') = (\sigma_{k+1}, \sigma_{k+1}, \ldots, \sigma_m).
$$
Then
$$
\Sigma^n_m(s) = (\sigma_{n+1}, \sigma_{n+2}, \ldots, \sigma_m).
$$
\end{lem}

\begin{lem}\label{comb spread}
Let $\Sigma^n_{n+4}(s) = (\sigma_{n+1},\sigma_{n+2}, \sigma_{n+3}, \sigma_{n+4})$. Then either
$$
\sigma_{n+3} = \sigma_{n+4} = (0, 0),
$$
or
$$
g^{\Sigma^n_{n+4}(s)}(J^-_{n+4}) \subset (-q_n-q_{n+5}, -q_{n+1}-q_{n+4}-q_{n+5})_c \Subset J^-_n.
$$
In the latter case, we have
$$
J^-_{n+4} \Subset (q_{n+5}-q_{n+4}, q_{n+4})_c \subset J^-_n  \cap g^{-\Sigma^n_{n+4}(s)}(J^-_n).
$$
\end{lem}

\begin{proof}
For concreteness, assume that $c_{-q_n}$ and $c_{-q_{n+1}}$ are the left and right endpoints of $J^-_n$ respectively.

Consider the partition of $J^-_n$ by orbit of the arcs $I^-_{n+4}$ and $I^-_{n+5}$. It is not hard to see that the leftmost and the rightmost arcs are $g_n(I^-_{n+5})$ and  $g_{n+1}(I^-_{n+4})$ respectively, and all other arcs are contained in between these two arcs. By the uniqueness of combinatorial addresses given in \lemref{codes}, the first claim follows.

Suppose that the latter case is true. Denote
$$
g^{\Sigma^n_{n+4}(s)}(J^-_{n+4}) = (-m_-, -m_+)_c
$$
for some $m_\pm \in \bbN$. Write
$$
g^{\Sigma^n_{n+4}(s)}(J^-_n) \cap J^-_n  = I_- \sqcup (-m_-, -m_+)_c \sqcup I_+,
$$
where
$$
I_- \supset (-m_- +q_{n+5}, -m_-]_c
\matsp{and}
I_+ \supset [-m_+, -m_+ +q_{n+4} + q_{n+5})_c.
$$
Then
$$
J^-_n  \cap g^{-\Sigma^n_{n+4}(s)}(J^-_n) \supset g^{-\Sigma^n_{n+4}(s)}(I_-) \sqcup J^-_{n+4} \sqcup g^{-\Sigma^n_{n+4}(s)}(I_+),
$$
where
$$
g^{-\Sigma^n_{n+4}(s)}(I_-) \supset (-q_{n+4} +q_{n+5}, -q_{n+4}]_c
\matsp{and}
g^{-\Sigma^n_{n+4}(s)}(I_+) \supset [-q_{n+5}, q_{n+4})_c.
$$
\end{proof}

\subsection{Pulling back a puzzle annulus to $\xi_s$}

Henceforth, we extend the domain of $g_k$ from $I^-_{k-1}$ to $\pbbD$, so that we have $g_k := g^{-q_k}$.

Let $n_0 \in \bbN$ be the number given in \lemref{initial scale}. For concreteness, assume that $n_0$ is even, so that $n_0 = 2\bn_0$. For $n \geq n_0$, let $\xi_s \in J^-_n$, and assume that $\xi_s$ is not an iterated preimage of $\xi_0$. Let $\Sigma^n_{n+4}(s) = (\sigma_{n+1},\sigma_{n+2}, \sigma_{n+3}, \sigma_{n+4})$ be the $(n, 4)$th combinatorial address of $s$, and suppose that either $\sigma_{n+3}$ or $\sigma_{n+4}$ is not equal to $(0,0)$. Define
$$
\hJ_{n+4}(s) := g^{-\Sigma^n_{n+4}(s)}(J^-_n) \cap J^-_n.
$$
By \lemref{comb spread}, we have $J^-_{n+4} \Subset \hJ_{n+4}(s)$. Let $R_n(s) \geq 1$ be the number such that $g^{R_n(s)} = g^{-\Sigma^n_{n+4}(s)}$, and let $V^n(s)$ and $U^n(s)$ be the $R_n(s)$th pullback along $\pbbD$ of $D^n|_{\hJ_{n+4}(s)}$ and $D^{n+4}$ respectively.

\begin{lem}\label{shift time}
We have $R_n(s) < q_{n+5}$.
\end{lem}

\begin{proof}
For $n+1 \leq i \leq n+4$, write $\sigma_i = (\alpha_i, \beta_i)$, where $0\leq \alpha_i < a_i$ and $\beta_i \in \{0,1\}$. Recall that
$$
g_i := g^{-q_i}
\matsp{and}
g^{\sigma_i} := g_i^{\alpha_i}\circ g_{i-1}^{\beta_i}.
$$
Since
$$
\alpha_iq_i + \beta_i q_{i-1} \leq q_{i+1} - q_i,
$$
the result follows.
\end{proof}

\begin{prop}\label{shifted annuli}
We have $X_s \subset U^n(s) \Subset V^n(s) \subset D^n$.
\end{prop}

\begin{proof}
The first inclusion is immediate from \propref{fiber in disks}. By \propref{nested disks} and \lemref{comb spread}, we have $D^{n+4} \Subset D^n|_{\hJ_{n+4}}$. Thus, $U^{n+4} \Subset V^{n+4}$. The last inclusion follows from \propref{disk pull in}.
\end{proof}

Define
\begin{equation}\label{eq:shifted annulus}
A^n(s) := V^n(s) \setminus \overline{U^n(s)}.
\end{equation}
By \propref{shifted annuli}, $A^n(s)$ is a non-degenerate annulus surrounding $X_s$.

\begin{prop}\label{spread annu}
There exists $\epsilon >0$ independent of $n$ such that
$$
\mod(A^n(s)) > \epsilon.
$$
\end{prop}

\begin{proof}
Define
$$
A := D^n|_{\hJ_{n+4}(s)} \setminus \overline{D^{n+4}}.
$$
The modulus of $A$ is given by the extremal length of the following path family
$$
\Gamma := \Gamma_A(\partial D^{n+4}, \partial D^n \cup (J^-_n \setminus \hJ_{n+4}(s))).
$$
Let $\Gamma^\pbbD \subset \Gamma$ be the path family such that $\gamma \in \Gamma^\pbbD$ has one endpoint in $\partial D^{n+4}$ and the other endpoint in $J^-_n \setminus \hJ_{n+4}(s)$. Then
$$
\Gamma = \Gamma_A(\partial D^{n+4}, \partial D^n) \cup \Gamma^\pbbD.
$$

Let $\Lambda(n) > 0$ be the constant given in \lemref{big lambda}. By \corref{hermancor} and \lemref{comb spread}, there exist uniform constants $0 < \lambda < \Lambda(n)$ and $0 < \delta < \lambda$ such that
$$
\hJ_{n+4}(s) \subset J^-_n \Subset J^-_{n+4}[\lambda]
\matsp{and}
J^-_{n+4}[\delta] \Subset \hJ_{n+4}(s).
$$
Hence, $\Gamma^\pbbD$ overflows the path family $\Gamma^\pbbD(D^{n+4}, \lambda)$ defined in \eqref{eq:path to circle}. By \lemref{overflow} and \propref{length to circle}, there exists a uniform constant $\omega >0$ such that
$$
\cW(\Gamma^\pbbD) \leq \omega.
$$
Clearly,
$$
\cL(\Gamma_A(\partial D^{n+4}, \partial D^n)) \geq \mod A^n.
$$
\lemref{union width} implies that
$$
\cW(\Gamma) \leq \frac{1}{\mod A^n} + \omega.
$$
Since $\mod A^n$ has a uniform lower bound by \thmref{lc at crit}, we conclude that the same is true for $\cL(\Gamma)$.

The iterate $F^{R_n(s)}$ maps the nested disks $U(s) \Subset V(s)$ to $D^{n+4} \Subset D^n|_{\hJ_{n+4}(s)}$ as a branched cover. By \propref{crit in disk} and \lemref{shift time}, this happens with uniformly bounded degree. The result now follows from \lemref{pullback mod}.
\end{proof}

\subsection{Nested sequence of puzzle annuli pullbacks at $\xi_s$}

Let $\xi_s \in J^-_{n_0}$, and assume that $\xi_s$ is not an iterated preimage of $c_0$. For $n \geq n_0$, let
$$
\xi_{s_n} := g^{-\Sigma^{n_0}_n(s)}(\xi_s) \in J^-_n.
$$
Write
$$
\Sigma^n_m = (\sigma_{n+1}, \sigma_{n+2}, \ldots, \sigma_m) := \Sigma^n_m(s_n)
\matsp{for}
m > n \geq n_0.
$$
By \lemref{comb add}, this simplified notation is consistent for different values of $n$ and $m$. Let $\hn_0 \geq n_0$ be the largest even number such that $s_{\hn_0} = s_{n_0}$.

\begin{lem}\label{shift pullback times}
There exists an infinite sequence $\{n_i = 2\bn_i\}_{i=1}^\infty$ of even numbers such that
\begin{itemize}
\item $n_1 \in \{\hn_0, \hn_0 + 2\}$;
\item $n_{i+1} \geq n_i+4$ for $i \geq 1$;
\item for $k >1$, we have
$$
g^{\Sigma^{n_1}_{n_k+4}} = g^{\Sigma^{n_1}_{n_1+4}} \circ \ldots \circ g^{\Sigma^{n_k}_{n_k+4}};
\hspace{5mm} \text{and}
$$
\item for $i \geq 1$, either $\sigma_{n_i+3}$ or $\sigma_{n_i+4}$ is not equal to $(0,0)$.
\end{itemize}
\end{lem}

\begin{proof}
Let $m = 2\bm \geq \hn_0$ be an even number. Clearly, there exists a unique sequence of even numbers $\{n_i(m)\}_{i=1}^{k_m}$ for some $k_m \geq 1$ such that $n_1(m) \in \{\hn_0, \hn_0+2\}$, and
$$
g^{\Sigma^{n_1(m)}_m} = g^{\Sigma^{n_1(m)}_{n_1(m)+4}} \circ \ldots \circ g^{\Sigma^{n_{k_m}(m)}_{n_{k_m}(m)+4}}.
$$
If
$$
\sigma_{m+1} = \sigma_{m+2} = (0,0),
$$
then we have
$$
\{n_i(m+2)\}_{i=1}^{k_{m+2}} = \{n_i(m)\}_{i=1}^{k_m}.
$$
Otherwise,
$$
k_{m+2} = k_{m-2}+1,
$$
and
$$
\{n_i(m+2)\}_{i=1}^{k_{m+2}} = \{n_i(m-2)\}_{i=1}^{k_{m-2}} \cup \{m+2\}.
$$
Note that in the latter case, we may have $n_1(m+2) \neq n_1(m)$.

Since $\xi_s$ is not an iterated preimage of $\xi_0$, there must be infinitely many even numbers $m \geq \hn_0$ such that either $\sigma_{m+1}$ or $\sigma_{m+2}$ is not equal to $(0,0)$. It follows that for some $n_1 \in \{\hn_0, \hn_0 +2\}$, we have $n_1 = n_1(m)$ for infinitely many even numbers $m > n_1$.
\end{proof}

Let $\{n_i\}_{i=1}^\infty$ be the sequence of even numbers given in \lemref{shift pullback times}. For $i \geq 1$, let $R_{n_i} \geq 1$ be the number such that
$$
g^{R_{n_i}} = g^{-\Sigma^{n_i}_{n_i+4}} = g^{-\sigma_{n_i+4}} \circ \ldots \circ g^{-\sigma_{n_i+1}}.
$$
We also let $R_{n_0} \geq 0$ be the number such that  
$$
g^{R_{n_0}} = g^{-\Sigma^{n_0}_{n_1}} = g^{-\Sigma^{\hn_0}_{n_1}}.
$$

\begin{lem}\label{bound on spread}
Let
$$
\bfR_{n_k} = \sum_{i=0}^k R_{n_i}.
$$
Then $\bfR_{n_k} \leq q_{n_k+9}$.
\end{lem}

\begin{proof}
If $g^r = g^{-\sigma_n}$, then $r < q_{n+2}$. Thus,
$$
R_{n_i} < q_{n_i +3} + q_{n_i +4} + q_{n_i +5} + q_{n_i +6} = r_{n_i + 3} + r_{n_i+5}
\matsp{for}
i \geq 1.
$$

If $n_1 = \hn_0$, then $R_{n_0} = 0$. Otherwise, $n_1 = \hn_0+2$, and
$$
\Sigma^{\hn_0}_{n_1} = (\sigma_{n_1-1}, \sigma_{n_1}).
$$
In either case, we have
$$
R_{n_0} < q_{n_1+1}+q_{n_1+2} = r_{n_1+1}.
$$

Since $n_{i+1} \geq n_i +4$ for $i \geq 1$, we have
$$
\bfR_{n_k} < \bfr_{n_k +5} < q_{n_k+9}
$$
by \lemref{exp time growth} i) and iv).
\end{proof}

\begin{thm}\label{spread}
Let $\xi_s \in J^-_{n_0}$, and assume that $\xi_s$ is not an iterated preimage of $\xi_0$. Then the fiber $X_s$ rooted at $\xi_s$ is trivial.
\end{thm}

\begin{proof}
For $i \geq 1$, denote
$$
\xi_{s_{n_i}} := g^{\bfR_{n_i}}(\xi_s) \in J^-_{n_i+4}.
$$
Consider the annulus
$$
A^{n_i}(s_{n_i}) = V^{n_i}(s_{n_i}) \setminus \overline{U^{n_i}(s_{n_i})}
$$
surrounding $X_{s_{n_i}}$ (see \eqref{eq:shifted annulus}). Let $\hV^{n_i}(s)$ and $\hU^{n_i}(s)$ be the $\bfR_{n_i}$th pullbacks of $V^{n_i}(s_{n_i})$ and $U^{n_i}(s_{n_i})$ along $\pbbD$. Then the annulus
$$
\hA^{n_i}(s) := \hV^{n_i}(s) \setminus \overline{\hU^{n_i}(s)}
$$
surrounds $X_s$. Moreover, by Proposition \ref{spread annu} and \ref{crit in disk}, and Lemma \ref{shift time} and \ref{pullback mod}, we see that $\mod(\hA^{n_i}(s))$ has uniform lower bound. The result now follows from \lemref{grot}.
\end{proof}

By combining \thmref{lc at crit} and \thmref{spread}, we obtain the following result.

\begin{cor}\label{last cor}
For $s \in \bbR/\bbZ$, the fiber $X_s$ rooted at $\xi_s \in \pbbD$ is trivial.
\end{cor}

Theorem A stated in \secref{sec:intro} now follows from \corref{blaschke instead}, \propref{fiber to lc} and \corref{last cor}.

\section{Nonexistence of Bounded Type Siegel Hedgehogs}

Let $f : \hbbC \to \hbbC$ be a polynomial of degree $d \geq 2$ that has a Siegel disk $\Delta_f$ containing a Siegel fixed point $0$ of bounded rotation number $\rho$. Let $\hDelta_f$ be a Siegel continuum containing $\overline{\Delta_f}$. Recall that this means $\hDelta_f$ is a compact connected set that maps bijectively into itself by $f$. To prove Theorem B stated in \secref{sec:intro}, we need to show that $\hDelta_f = \overline{\Delta_f}$.

By \thmref{blaschke model}, there exists a Blaschke product model $F \in \cH^d_\rho$, the modified Blaschke product $\tiF$ obtained from $F$, and a quasiconformal map $\eta : \hbbC \to \hbbC$ such that
$$
f = \eta \circ \tiF \circ \eta^{-1}.
$$
Denote
$$
\hDelta_f^\partial := \hDelta_f \setminus \Delta_f
\matsp{and}
\chDelta_f^\partial := \eta^{-1}(\hDelta_f^\partial).
$$
Since $\eta$ maps $\overline{\bbD}$ homeomorphically to $\overline{\Delta_f}$, it suffices to show that $\chDelta_f^\partial \subset \partial \bbD$.

Consider the puzzle neighborhood $\bfP^n$ of $\pbbD$ of depth $n \geq 0$ defined in \eqref{eq:puzzle nbh}. By \corref{last cor}, we have
\begin{equation}\label{eq:circle fiber}
\bigcap_{n=0}^\infty \bfP^n = \bigcup_{s \in \bbR/\bbZ} X_s = \pbbD.
\end{equation}
Hence, there exists $N > 0$ such that for $n \geq N$, the set $\overline{\bfP^n}$ does not contain any fixed points of $F$.

The boundary of $\bfP^n$ is the union of arcs in iterated preimages of $\pbbD$, iterated preimages of fixed points (the landing points of bubble rays in the puzzle partition), and external rays. If $\chDelta_f^\partial$ crossed the boundary of $\bfP^n$, then this would contradict the fact that $\chDelta_f^\partial$ maps bijectively into itself. Hence, $\chDelta_f^\partial \subset \bfP^n$ for all $n \geq N$. Theorem B now follows from \eqref{eq:circle fiber}.


\end{document}